\documentclass{amsart}

\usepackage[title]{appendix}
\usepackage{amssymb}
\usepackage{amsmath}
\usepackage{amsthm}
\usepackage{tikz}
\usepackage{tikz-cd}
\usepackage{mathtools}
\usepackage{framed}
\usepackage{comment}
\usepackage{hyperref}
\usepackage[capitalize]{cleveref}

\usepackage[margin=30mm]{geometry}

\newcommand{\Q}{\mathbb{Q}}
\newcommand{\R}{\mathbb{R}}

\newcommand{\Z}{\mathbb{Z}}

\newcommand{\F}{\mathbb{F}}

\newcommand{\Alg}{\text{\normalfont{Alg}}}
\newcommand{\bs}{\boldsymbol}

\renewcommand{\mod}{\,\,\text{\normalfont mod}\,}
\newcommand{\M}{\mathscr{M}}

\DeclareMathOperator{\Aut}{Aut}

\DeclareMathOperator{\Hom}{Hom}
\DeclareMathOperator{\id}{id}

\DeclareMathOperator{\cd}{cd}

\DeclareMathOperator{\rank}{rank}

\DeclareMathOperator{\Def}{def}
\DeclareMathOperator{\IM}{Im}

\DeclareMathOperator{\Ker}{Ker}

\DeclareMathOperator{\Diff}{Diff}

\DeclareMathOperator{\SF}{SF}

\DeclareMathOperator{\HT}{HT}
\DeclareMathOperator{\PHT}{PHT}
\DeclareMathOperator{\gd}{gd}

\DeclareMathOperator{\FL}{FL}

\DeclareMathOperator{\st}{st}

\DeclareMathOperator{\sCob}{sCob}

\DeclareMathOperator{\mst}{/\hspace{-1mm}=_{\st}}
\DeclareMathOperator{\dcong}{\cong_{\Diff}}
\DeclareMathOperator{\ev}{ev}

\newtheorem{thm}{Theorem}[section]

\newtheorem*{thm*}{Theorem}
\newtheorem{prop}[thm]{Proposition}
\newtheorem*{prop*}{Proposition}
\newtheorem{lemma}[thm]{Lemma}
\newtheorem{corollary}[thm]{Corollary}
\newtheorem*{corollary*}{Corollary}

\newtheorem*{question*}{Question}
\newtheorem{problem}[thm]{Problem}
\newtheorem*{problem*}{Problem}
\newtheorem{plist}{Problem}

\newtheorem{thmx}{Theorem}

\theoremstyle{definition}

\newtheorem*{theorem*}{Theorem}
\newtheorem{construction}[thm]{Construction}

\theoremstyle{remark}
\newtheorem{remark}[thm]{Remark}
\newtheorem*{remark*}{Remark}

\newtheoremstyle{custom}{}{}{\itshape}{}{\bfseries}{.}{.5em}{\thmnote{#3}#1}
\theoremstyle{custom}

\usepackage{rotating}

\usepackage{color}
\usepackage{enumerate}

\usepackage{cite}
\usepackage[nodisplayskipstretch]{setspace}
\usepackage{placeins}

\usepackage{listings}
\lstdefinestyle{mystyle}{
    basicstyle=\ttfamily\footnotesize,
    breakatwhitespace=false,         
    breaklines=true,                 
    captionpos=b,                    
    keepspaces=true,                                   
    showspaces=false,                
    showstringspaces=false,
    showtabs=false,                  
    tabsize=2
}
\newenvironment{clist}[1]
{\begin{enumerate}[\normalfont #1]}
{\end{enumerate}}

\newenvironment{sss}[1]
{
\medskip
\setcounter{subsubsection}{\value{subsubsection}+1}
\paragraph{\arabic{section}.\arabic{subsection}.\arabic{subsubsection}. \emph{#1}}
}

\newcommand{\wh}{\widehat}
\newcommand{\wt}{\widetilde}

\usepackage{mathrsfs}
\usepackage{adjustbox}

\usepackage{nccmath}

\makeatletter
\@namedef{subjclassname@2020}{%
  \textup{2020} Mathematics Subject Classification}
\makeatother

\usetikzlibrary{decorations.markings}

\begin{document}

\title[Stably free modules and the unstable classification of 2-complexes]{Stably free modules and the unstable classification \\ of 2-complexes}

\author{John Nicholson}
\address{School of Mathematics and Statistics, University of Glasgow, U.K.}
\email{john.nicholson@glasgow.ac.uk}



\begin{abstract}
For all $k \ge 2$,  we show that there exists a group $G$ and a non-free stably free $\Z G$-module of rank $k$. We use this to show that, for all $k \ge 2$, there exist homotopically distinct finite $2$-complexes with fundamental group $G$ and with Euler characteristic exceeding the minimal value over $G$ by $k$.
This resolves Problem D5 in the 1979 Problem List of C. T. C. Wall.
We also explore a number of generalisations and present a potential application to the topology of closed smooth 4-manifolds.
\end{abstract}

\maketitle

\section{Introduction}

An important concept in algebra and topology is that of stabilisation. When faced with an intractable classification problem, it is often possible to classify up to a weaker notion of stable equivalence.
For example, a projective $R$-module can be `stabilised' by taking a direct sum with the free module $R$, and we say that two projective $R$-modules $P$ and $Q$ are (reduced) \textit{stably equivalent} if they are related by a sequence of stabilisations, i.e. if $P \oplus R^n \cong Q \oplus R^m$ for some $n,m \ge 0$. 
While it is typically difficult to classify the finitely generated projective $R$-modules up to $R$-isomorphism, the set of stable equivalence classes form a group $\wt K_0(R)$ which is computable, at least in principle, using the methods of algebraic $K$-theory. Given such a stable classification, it then remains to classify the objects within a stable class up to isomorphism.  
This is the corresponding \textit{unstable classification}.

The aim of this article will be to study the unstable classification of three closely related objects, where $G$ is a fixed group and $\rightsquigarrow$ denotes the stabilisation:
\begin{clist}{(I)}
\item Finitely generated (non-zero) projective $\Z G$-modules up to $\Z G$-isomorphism, with $P \rightsquigarrow P \oplus \Z G$.
\item Finite 2-complexes $X$ with $\pi_1(X) \cong G$ up to homotopy equivalence, with $X \rightsquigarrow X \vee S^2$.
\item Closed smooth 4-manifolds $M$ with $\pi_1(M) \cong G$ up to homeomorphism, with $M \rightsquigarrow M \# (S^2 \times S^2)$.
\end{clist}
In each case, the stable classification can be computed in principle. For (I), the stable classification is given by $\wt K_0(\Z G)$, as above. 
For (II), the stable classification is trivial, i.e. for all $X$, $Y$ we have $X \vee nS^2 \simeq Y \vee mS^2$ for some $n,m \ge 0$ \cite[Theorem 13]{Wh39}. 
For (III), the stable classification can be reduced to the computation of certain bordism groups using Kreck's modified surgery \cite{Kr99}. 

Our main result is the construction of examples which illustrate qualitative features in the unstable classifications of projective $\Z G$-modules (I) and finite 2-complexes (II). We will show:

\begin{thmx} \label{thm:simple-modules}
For all $k \ge 2$, there exists a group $G$ and finitely generated projective $\Z G$-modules $P$ and $Q$ such that $P \oplus \Z G \cong Q \oplus \Z G$ and $Q \cong Q_0 \oplus \Z G^k$ for some $\Z G$-module $Q_0$, but $P \not \cong Q$.
\end{thmx}

\begin{thmx} \label{thm:simple-complexes}
For all $k \ge 2$, there exists finite $2$-complexes $X$ and $Y$ such that $X \vee S^2 \simeq Y \vee S^2$ and $Y \simeq Y_0 \vee kS^2$ for some finite $2$-complex $Y_0$, but $X \not \simeq Y$.
\end{thmx}

In each case, for $k \ge 2$, we can take $G = \ast_{i=1}^k T$ to be the free product of $k$ copies of the trefoil group $T= \langle x,y \mid x^2=y^3 \rangle$. The examples in \cref{thm:simple-modules} are in the simplest case where $Q = \Z G^k$ and so $P$ is a non-free stably free $\Z G$-module of rank $k$. The examples $X$, $Y$ in \cref{thm:simple-complexes} are distinguished by showing that $\pi_2(X) \cong P$ and $\pi_2(Y) \cong \Z G^k$ as $\Z G$-modules.

We will now review unstable classification for (I)-(III) before discussing each case individually. We give generalisations of Theorems \ref{thm:simple-modules} and \ref{thm:simple-complexes} and applications to smooth 4-manifolds (III).

\subsection{Background on unstable classification} \label{ss:intro-unstable}

For a set $\mathcal{S}$, define a \textit{stabilisation} to be a function $\Sigma : \mathcal{S} \to \mathcal{S}$ such that $\ell(a) := \sup\{k \ge 0: \Sigma^n(a) = \Sigma^{n+k}(b) \text{ for some $b \in \mathcal{S}$, $n \ge 0$}\} < \infty$ where $n, k$ are integers. That is, for all $a \in \mathcal{S}$, there does not exist an infinite sequence $b_k \in \mathcal{S}$ for $k \ge 1$ such that, for each $k$, we have $\Sigma^n(a)=\Sigma^{n+k}(b_k)$ for some $n$.
The stabilisations in (I)-(III) each satisfy this condition (see \cref{section:unstable}).
The function $\Sigma$ is then fixed point free and we have a \textit{level function} $\ell : \mathcal{S} \to \Z_{\ge 0}$ which is surjective and satisfies $\ell(\Sigma(a)) = \ell(a)+1$ for $a \in \mathcal{S}$. For $n \ge 0$, we say that $a \in \mathcal{S}$ has \textit{level $n$} if $\ell(a)=n$ and is at the \textit{minimal level} if $\ell(a)=0$.

We can view $(\mathcal{S},\Sigma)$ as a directed graph with vertex set $\mathcal{S}$ and edges between $a$ and $\Sigma(a)$ for each $a \in \mathcal{S}$. Let $=_{\st}$ denote the corresponding (reduced) stable equivalence relation, i.e. $a =_{\st} b$ if $\Sigma^n(a) = \Sigma^m(b)$ for some $n,m \ge 0$. For each stable equivalence class $c \in \mathcal{S} \mst$, the subgraph $\mathcal{S}_c = \{ a \in \mathcal{S} : a =_{st} c\}$ is a directed tree which is graded by $\ell_c := \ell \mid_{\mathcal{S}_c} : \mathcal{S}_c \to \Z_{\ge 0}$.
The goal of the unstable classification is then to determine the trees $\mathcal{S}_c$ for each $c \in S \mst$.

Whilst determining the trees $\mathcal{S}_c$ explicitly will often be intractable, further progress can be made by asking for structural features of the trees.
For $k \ge 0$, we say that $\mathcal{S}_c$ has \textit{cancellation at level $k$} if $\Sigma(a)=\Sigma(b)$ implies $a=b$ for all $a,b \in \mathcal{S}_c$ with $\ell(a)=\ell(b) \ge k$, or equivalently if $|\ell_c^{-1}(n)|=1$ for all $n \ge k$ (see \cref{figure:diagrams}c). We say that $\mathcal{S}$ has cancellation at level $k$ if this holds for all $c \in \mathcal{S} \mst$. If such a $k$ exists, then we have a kind of `stable range' where elements $a \in \mathcal{S}$ with $\ell(a) \ge k$ are classified by $\ell(a) \in \Z_{\ge 0}$ and the stable class $c$.
We say that $\mathcal{S}_c$ (resp. $\mathcal{S}$) has \textit{cancellation} if it has cancellation at level $0$, or equivalently if $\ell_c$ (resp. $\ell$) is bijective (see \cref{figure:diagrams}a).

\begin{figure}[h] \vspace{-4mm} 
\begin{center}
\begin{tabular}{ccccccccc}
\begin{tabular}{l}
\begin{tikzpicture}
\draw[fill=black] (2,0) circle (2pt);
\draw[fill=black] (2,1) circle (2pt);
\draw[fill=black] (2,2) circle (2pt);
\draw[fill=black] (2,3) circle (2pt);

\node at (2,3.8) {$\vdots$};
\node at (1.2,3) {(a)};
\node at (3,0) { };

\draw[black] 
(-1.5,0) node[right]{$\ell = 0$}
(-1.5,1) node[right]{$\ell = 1$}
(-1.5,2) node[right]{$\ell = 2$}
(-1.5,3) node[right]{$\ell = 3$};
\node at (-1,3.8) {$\vdots$};

\draw[thick]  (2,0) -- (2,1) -- (2,2) -- (2,3);
\end{tikzpicture} 
\end{tabular}
&&
\begin{tabular}{l}
\begin{tikzpicture}
\draw[white] (1,0) node[right]{$\ell = 0$};

\draw[fill=black] (1,0) circle (2pt);
\draw[fill=black] (2,0) circle (2pt);
\draw[fill=black] (3,0) circle (2pt);
\draw[fill=black] (2,1) circle (2pt);
\draw[fill=black] (2,2) circle (2pt);
\draw[fill=black] (2,3) circle (2pt);

\node at (2,3.8) {$\vdots$};
\node at (1.2,3) {(b)};

\draw[thick] (1,0) -- (2,1) -- (3,0) (2,0) -- (2,1) -- (2,2) -- (2,3);
\end{tikzpicture} 
\end{tabular}
&&&&	
\begin{tabular}{l}
\begin{tikzpicture}
\draw[white] (1,0) node[right]{$\ell = 0$};

\draw[fill=black] (-0.25,0) circle (2pt);
\draw[fill=black] (0.5,0) circle (2pt);
\draw[fill=black] (2,0) circle (2pt);
\draw[fill=black] (2,1) circle (2pt);
\draw[fill=black] (2,2) circle (2pt);
\draw[fill=black] (2,3) circle (2pt);
\draw[fill=black] (1.25,0) circle (2pt);
\draw[fill=black] (1.25,1) circle (2pt);
\draw[fill=black] (0.5,1) circle (2pt);

\node at (1.2,3) {(c)};
\node at (2,3.8) {$\vdots$};
\draw[black] (2,1) node[right]{
$\left.
    \begin{array}{ll}
           \\ \\ \\
    \end{array}
\right \} k$};
\draw[thick] (-0.25,0) -- (0.5,1)--(2,2) (0.5,0) -- (0.5,1) (2,0) -- (2,1) 
(2,1) -- (2,2) -- (2,3) (1.25,0)--(2,1) (1.25,1)--(2,2);
\end{tikzpicture}
\end{tabular}
\end{tabular}
\end{center}
\caption{Branching phenomena for the graded trees $\mathcal{S}_c$. The vertical height is $\ell(\cdot) \in \Z_{\ge 0}$.} \label{figure:diagrams}
\vspace{-2mm}
\end{figure}

For a group $G$, let $P(\Z G)$, $\HT(G)$ and $\M(G)$ denote the sets with stabilisation described in (I), (II) and (III) respectively. 
Each of these sets has cancellation when $G=\{1\}$ is the trivial group. For $P(\Z)$, this is trivial. For $\HT(\{1\})$ this is an exercise in homotopy theory \cite[Theorem 1]{MW50}. 
For $\M(\{1\})$, this follows from deep results of Donaldson \cite{Do83} and Freedman \cite{Fr82} (see \cite[p14]{KS84}).

When $G$ is a finite group, $P(\Z G)$, $\HT(G)$ and $\M(G)$ each have cancellation at level one (see \cref{figure:diagrams}b).
This was shown for $P(\Z G)$ by Swan \cite{Sw60}, for $\HT(G)$ by Browning \cite{Br78} and for $\M(G)$ by Hambleton-Kreck \cite{HK93-II}. The three proofs share a number of techniques; 
in fact, Hambleton-Kreck gave a new proof of case (II) before using these techniques as the basis for their proof of case (III) (see \cite{HK93-I}).
On the other hand, examples have been given to illustrate that cancellation (at level 0) fails for an arbitrary finite group $G$. This was shown for $P(\Z Q_{32})$ by Swan \cite{Sw62}, where $Q_{32}$ is the quaternion group of order 32, by Swan \cite{Sw62}, for $\HT((\Z/5)^3)$ by Metzler \cite{Me76} and for $\M((\Z/5)^3)$ by Kreck-Schafer \cite{KS84} by applying the boundary of thickenings construction (see \cref{ss:intro-manifolds}) to the examples of Metzler.

Much less is known about the case where $G$ is infinite.
In particular, it has previously not been known whether there exists $G$ such that $P(\Z G)$ fails cancellation at any level $\ge 1$, $\HT(G)$ fails cancellations at any level $\ge 2$, or $\M(G)$ fails cancellation at any level $\ge 1$.
We will now discuss each case in turn.

\subsection{Projective $\Z G$-modules} \label{ss:intro-SF}

The \textit{rank} of a finitely generated projective $\Z G$-module $P$, denoted by $\rank_{\Z G}(P)$, is the rank of the free abelian group $P_G = \Z \otimes_{\Z G} P$. 
We will focus on the case of stably free $\Z G$-modules $S$ where $\rank_{\Z G}(S) = m-n$ for any $n,m \ge 0$ such that $S \oplus \Z G^n \cong \Z G^{m}$. Let $\SF(\Z G) \subseteq P(\Z G)$ denote the subset of the stably free $\Z G$-modules.
If $\ell = \ell_{\SF(\Z G)}$ is the level function on $P(\Z G)$ restricted to $\SF(\Z G)$, then $\ell(S) = \rank_{\Z G}(S)-1$ for $S \in \SF(\Z G)$ (\cref{prop:prelim-P(ZG)}).

\sss{Main result} \label{sss:main-modules}

Examples of non-free stably free $\Z G$-modules of rank one have been constructed over various torsion free groups \cite{Du72,BD79,Ar81,HJ06a} and groups of the form $G=F_n \times H$ for $H$ finite \cite{Jo12a, Ka10,Os12}.
However, there has previously been no known examples of (a) a non-free stably free $\Z G$-module of rank $\ge 2$ over a group $G$ (see \cite[p623]{Jo12b}, \cite[p.xiii]{Jo12a}), or (b) a non-free projective $\Z G$-module of rank $\ge 2$ over a torsion free group $G$ (see \cite[p2950]{AB89}).
For each $k \ge 2$, \cref{thm:simple-modules} is the statement that there exists $G$ for which $P(\Z G)$ fails cancellation at level $k$.
Since $G = \ast_{i=1}^k T$ is torsion free, this gives examples for both (a) and (b) above.

We now state a more detailed and general version of \cref{thm:simple-modules}, motivated by applications to 2-complexes in \cref{thm:main}.
Let $\cd(G)$ denote the cohomological dimension of $G$ and recall that $\Z G$-modules $M$, $N$ are \textit{$\Aut(G)$-isomorphic} if $M \cong N_\theta$ are $\Z G$-isomorphic for some $\theta \in \Aut(G)$, where $N_\theta$ is the abelian group $N$ equipped with the $G$-action $g \cdot x := \theta(g) \cdot_N x$ for $g \in G$ and $x \in N$.

\begingroup
\renewcommand\thethm{A$'$}
\begin{thm} \label{thm:main-SF}
For all $k \ge 1$ and $d \ge 2$, there exists a finitely presented group $G$ with $\cd(G) = d$ and infinitely many stably free $\Z G$-modules $S_i$ of rank $k$ which are distinct up to $\Aut(G)$-isomorphism and such that, for all $i$, $S_i \not \cong S \oplus \Z G$ for any $\Z G$-module $S$.
\end{thm}
\setcounter{thm}{\value{thm}-1}
\endgroup

When $d=2$, we take $G = \ast_{i=1}^k T$ where $T$ is the trefoil group. Here the case $k=1$ was shown by Berridge-Dunwoody \cite{BD79}. For $d \ge 3$, the groups are constructed by applying the operation
\[ G \rightsquigarrow (G \ast \langle q \mid - \rangle) \ast_{\langle q=r^2 \rangle} \langle r \mid - \rangle.\]
to $\ast_{i=1}^k T$ a total of $(d-2)$ times.

We will now explain the key idea behind our proof in the $d=2$ case, i.e. when $G = \ast_{i=1}^k T$.
Let $\F$ be a field and let $G = \ast_{i=1}^k G_i$ be a free product of group. We say that an $\F G$-module $M$ is \textit{induced} if there exists $\F G_i$-modules $M_i$ and an $\F G$-module isomorphism
\[ M \cong {\iota_1}_\#(M_1) \oplus \cdots \oplus {\iota_k}_\#(M_k)\]
where ${\iota_i}_\#(M_i) = \F G \otimes_{\F G_i}  M_i$ is extension of scalars for the inclusion 
$\iota_i : G_i \hookrightarrow G$.
It follows from Bergman's theorem on modules over coproducts of rings 
\cite{Be74} that, if $M$ has no $\F G$ summand, then $M$ is uniquely determined by the $M_i$ (\cref{cor:Bergman}). 

For $G = \ast_{i=1}^k T$, we will show that the $\Z G$-modules $S_i$ of \cref{thm:main-SF} are distinct by showing that the corresponding $\F_p[\ast_{i=1}^k T/T'']$-modules $\F_p[\ast_{i=1}^k T/T''] \otimes_{\Z G} S_i$ are distinct, where $\F_p$ is the finite field of characteristic $p$ and $T''$ is the second derived subgroup. This is achieved using Bergman's theorem. To show that the modules have no $\F_p[\ast_{i=1}^k T/T'']$ summand, we use that the group ring $\F_p[T/T'']$ is stably finite (see \cref{subsection:stably-finite}) since $T/T''$ is polycyclic and so is a sofic group \cite{ES04}.
This strategy was proposed by Evans in \cite{Ev99}, though an example was never given.

\sss{Cancellation bounds}

If $\Z G$ is Noetherian of Krull dimension $d_G$ and $d = d_G +1$, then stably free $\Z G$-modules of rank $\ge d$ are free \cite[Chapter IV]{Ba68}.
If $G$ is polycyclic-by-finite, then $\Z G$ is Noetherian and it is conjectured that these are the only such groups (see \cite[p328]{KL19}). 
If $\Z G$ is not Noetherian, then such a bound $d$ can often still be found; for example, if $G$ is a free group, then we can take $d=0$ \cite{Ba64}. This raises the question of whether, given a group $G$, there always exist a bound $d$ such that every stably free $\Z G$-module of rank $\ge d$ is free. We will show that there does not.

\begin{thmx} \label{thm:main-SF-further}
There exists a group $G$ such that, for all $k \ge 1$, there is a non-free stably free $\Z G$-module of rank $k$.
\end{thmx}

We construct examples over the non-finitely generated group $G= \ast_{i=1}^\infty T$. 

\subsection{Finite $2$-complexes} \label{ss:intro-complexes}

A \textit{finite $2$-complex} will be taken to mean a connected finite 2-dimensional CW-complex.
Let $G$ be finitely presented and let $\ell = \ell_{\HT(G)}$. Then $\ell(X) = \chi(X) -\chi_{\min}(G)$ for all $X \in \HT(G)$, where $\chi_{\min}(G):= \min\{\chi(X) : X \in \HT(G)\}$ (\cref{prop:prelim-HT(G)}).

Recall that a finite presentation $\mathcal{P}$ for a group $G$ has an associated \textit{presentation complex} $X_{\mathcal{P}}$ which is a finite 2-complex with $\pi_1(X_{\mathcal{P}})\cong G$. Conversely, every finite 2-complex $X$ with $\pi_1(X) \cong G$ is homotopy equivalent to $X_{\mathcal{P}}$ for some finite presentation $\mathcal{P}$ for $G$ (see, for example, \cite[p61]{HMS93}).

\sss{Stably free modules and $2$-complexes}

We will now discuss the link between stably free $\Z G$-modules and finite 2-complexes. This is the basis for our strategy for proving \cref{thm:main} (see below).

If $X$ is a finite 2-complex with $\pi_1(X) \cong G$, then $\pi_2(X)$ can be viewed as a $\Z G$-module via the isomorphism $\pi_2(X) \cong \pi_2(\wt X)$ and the monodromy action. This has the property that $\pi_2(X \vee S^2) \cong \pi_2(X) \oplus \Z G$ as $\Z G$-modules.
If $\gd(G)=2$, then $\pi_2(X)$ is a stably free $\Z G$-module of rank $\ell(X) \in \Z_{\ge 0}$ (\cref{prop:syzygies-cd-finite}).
In particular, there is a level preserving map
\[\pi_2 : \HT(G) \to \SF(\Z G) \cup \{0\}\]
and $\ell_{\SF(\Z G)}(\pi_2(X)) = \ell_{\HT(G)}(X)-1$, where we take $\ell_{\SF(\Z G)}(\{0\}):=-1$. The trefoil group $T$ has $\gd(T)=2$ and it was shown by Dunwoody \cite{Du76} that $\HT(T)$ fails cancellation at level one using a non-free stably free $\Z T$-module of rank one \cite{Du72}.

\sss{Main result}

The question of whether, for each $k \ge 2$, there exist a group $G$ such that $\HT(G)$ fails cancellation at level $k$ has been raised in a number of variants (see \cite[Problem C]{Dy79a} and \cite[p124]{HMS93}).
Most notably, the following version appeared in the 1979 Problems List of C. T. C. Wall \cite{Wa79}. If such an $X$ exists then, if $X_0$ is a finite 2-complex such that $\pi_1(X_0) \cong \pi_1(X)$ and $\ell(X_0)=0$, then $X$ and $X_0 \vee kS^2$ are homotopy distinct finite 2-complexes at level $k$ in $\HT(\pi_1(X))$.

\begin{problem}[Problem D5 from Wall's list \cite{Wa79}]
\label{problem:2-complexes}
For each $k \ge 2$, does there exists a finite $2$-complex $X$ such that $\ell(X) = k$ and $X \not \simeq Y \vee S^2$ for any finite $2$-complex $Y$?
\end{problem}

\begin{remark}
Wall's list contains eight problems concerning 2-complexes \cite[List D]{Wa79}. Some are classical, some are due to Wall and others were suggested by participants at the 1977 Durham Symposium on Homological Group Theory.
Each problem asks whether examples exist which illustrate certain phenomena. The only examples previously found were the finite 2-complexes of Metzler \cite{Me90} (see also \cite{Lu91}) which are homotopy equivalent but not simple homotopy equivalent, resolving Problem D6. Bestvina-Brady \cite{BB97} showed that there is a counterexample to the Eilenberg-Ganea conjecture (Problem D4) or the Whitehead conjecture (Problem D7).
\end{remark}

Our main result is an affirmative answer to this question for all $k \ge 2$. We will also pursue the following natural generalisation to higher dimensions.
For $n \ge 2$, a \textit{$(G,n)$-complex} is an $n$-complex $X$ with $\pi_1(X) \cong G$ and such that the universal cover $\wt X$ is $(n-1)$-connected. Equivalently, $X$ is the $n$-skeleton of a $K(G,1)$. 
Let $\HT(G,n)$ denote the set of finite $(G,n)$-complexes up to homotopy. Then $X \rightsquigarrow X \vee S^n$ is a stabilisation (\cref{prop:prelim-HT(G)}).

The natural extension of \cref{problem:2-complexes} to $(G,n)$-complexes was considered by Dyer in \cite{Dy79a,Dy79b} (see \cite[p378]{Dy79b}). However, there were still no examples found at level $k \ge 2$. We will show:

\begingroup
\renewcommand\thethm{B$'$}
\begin{thm} \label{thm:main}
For all $n \ge 2$ and $k \ge 0$, there exists a group $G$ and  infinitely many homotopically distinct finite $(G,n)$-complexes $X_i$ at level $k$ such that $X_i \not \simeq Y \vee S^n$ for any finite $(G,n)$-complex $Y$.
\end{thm}
\setcounter{thm}{\value{thm}-1}
\endgroup

For $n = 2$ and $k \ge 1$, we take $G = \ast_{i=1}^k T$ and $X_i = \bigvee_{j=1}^k X_{\mathcal{P}_i}$ for $i \ge 1$ where
\[ \mathcal{P}_i = \langle x,y,a,b \mid x^2=y^3, a^2=b^3, x^{2i+1}=a^{2i+1}, y^{3i+1}=b^{3i+1} \rangle \]
are the presentations of Harlander-Jensen \cite{HJ06a}. We prove the $X_i$ are homotopically distinct by showing that the $\Z G$-modules $\pi_2(X_i)$, which are stably free, coincide with the examples behind \cref{thm:main-SF}.
For $k=0$, we take $G=(T \ast \langle q \mid - \rangle) \ast_{\langle q=r^2 \rangle} \langle r \mid - \rangle$ and $X_i = X_{\mathcal{Q}_i}$ where
\[ \mathcal{Q}_i = \langle a, b, c \mid a^2=b^3, [a^2,b^{2i+1}], [a^2,c^{3i+1}] \rangle \]
for the group  constructed by Lustig \cite{Lu93}.
For $n \ge 3$ and $k \ge 1$, the complexes are constructed from \cref{thm:main-SF} using Wall's theorem on the realisability of chain complexes in dimensions $\ge 3$ (see \cref{prop:realisation-thm}). The case $k=0$ requires a generalisation of the examples of Lustig (see \cref{subsection:min-EC}).

In \cref{cor:syzygies}, we will point out that \cref{thm:main} shows that, for each $n \ge 2$ and $k \ge 2$, there exists a finitely presented group $G$ such that the syzygies $\Omega_n^G(\Z)$ have non-cancellation at level $k \ge 0$. This resolves a problem raised by Johnson in \cite[p.xiii]{Jo12a}. 

\sss{Homotopy classification over free products}

The key ingredient in the proofs of Theorems \ref{thm:main-SF} and \ref{thm:main} is Bergman's theorem on $\F G$-modules for a field $\F$ and a group $G = \ast_{i=1}^k G_i$ (see \cref{ss:intro-SF}).
For a finite 2-complex $X$ with $\pi_1(X) \cong G$, this can be applied to determine the structure of $\pi_2(X) \otimes \F$ as an $\F G$-module (see Propositions \ref{prop:existence} and \ref{prop:uniqueness}).
This raises the question of whether a general classification of $P(\Z G)$ and $\HT(G)$ can be obtained for $G = \ast_{i=1}^k G_i$.
We will show that it cannot since, in general, we lose information by passing from $\Z G$-modules to $\F G$-modules. 

\begin{thm} \label{thm:main-free-product}
Let $k \ge 2$. Then:
\begin{clist}{(i)}
\item
There exists a group $G = \ast_{i=1}^k G_i$ and a finite $2$-complex $X$ with $\pi_1(X) \cong G$ such that
$\pi_2(X)$ is not an induced $\Z G$-module.
\item
There exists a group $G = \ast_{i=1}^k G_i$ and a finite $2$-complex $X$ with $\pi_1(X) \cong G$ such that
$\pi_2(X)$ has two induced module structures
\[ \pi_2(X) \cong {\iota_1}_\#(M_1) \oplus \cdots \oplus {\iota_k}_\#(M_k) \cong {\iota_1}_\#(M_1') \oplus \cdots \oplus {\iota_k}_\#(M_k') \]
such that, for each $i$, $M_i$ and $M_i'$ have no $\Z G_i$ summands and are not $\Aut(G_i)$-isomorphic.
\end{clist}
\end{thm}

For (i), we take $G = \ast_{i=1}^k (\Z/p_i)^2$ for distinct primes $p_i$ and our results are a minor extension of those of Hog-Angeloni--Lustig--Metzler \cite{HLM85}. For (ii), we take $G = \ast_{i=1}^k (\Z/p_i)^3$ for distinct primes $p_i$ such that $p_i \equiv 1 \mod 4$. These examples combine ideas from \cite{HLM85} with those of Metzler \cite{Me76}.

\subsection{Smooth $4$-manifolds} \label{ss:intro-manifolds}

A \textit{$4$-manifold} will be assumed to be closed, smooth and connected. Alongside $\M(G)$, we can also consider the set $\M^{\Diff}(G)$ of 4-manifolds $M$ with $\pi_1(M) \cong G$ up to diffeomorphism. 
Whilst $\M(\{1\})$ has cancellation, the existence of exotic smooth structures on simply connected 4-manifolds shows that $\M^{\Diff}(\{1\})$ fails cancellation. There are only a few examples where cancellation is known to fail for $\M(G)$ (see \cite[Sections 5.(4,7,10)]{KPR22}). 
On the other hand, for each $k \ge 1$, it is currently open whether there exists $G$ such that either $\M(G)$ or $\M^{\Diff}(G)$ fail cancellation at level $k$. The question for $\M^{\Diff}(G)$ was considered by Kreck \cite[p198]{Kr16} and Crowley \cite[Problem 10B]{BCD21}. 

The following is a potential application of stably free $\Z G$-modules to the unstable classification of 4-manifolds.
A $\Z G$-module $S$ is said to be \textit{geometrically realisable} if there exists a finite $2$-complex $X$ such that $\pi_1(X) \cong G$ and $\pi_2(X) \cong S$. Let $S^* = \Hom_{\Z G}(S,\Z G)$ denote the dual (see \cref{ss:duals}).

\begin{thm} \label{thm:S+S*-intro}
Let $G$ be a finitely presented group such that $\gd(G)=2$ and suppose there exists a stably free $\Z G$-module $S$ of rank $k$ which is geometrically realisable and such that $S \oplus S^*$ is not a free $\Z G$-module.
Then both $\M^{\Diff}(G)$ and $\M(G)$ fail cancellation at level $k$. 
\end{thm}

\begin{remark}
If $G$ is a finitely presented group such that $\gd(G)=2$, then every stably free $\Z G$-module is geometrically realisable if and only if $G$ has the D2 property (see \cref{section:Gn-complexes}). In particular, if a stably free $\Z G$-module $S$ of rank $k$ exists such that $S \oplus S^*$ is free, then either $G$ is a counterexample to Wall's D2 problem or both $\M^{\Diff}(G)$ and $\M(G)$ fail cancellation at level $k$.
\end{remark}

We do not know whether or not a stably free $\Z G$-module exists which satisfies these conditions, or even whether there exists any group $G$ and a stably free $\Z G$-module $S$ such that $S \oplus S^*$ is non-free (see \cref{problem:S+S*}).
Such examples would also provide further examples for \cref{thm:simple-modules} since $S \oplus S^*$ would be a stably free $\Z G$-module of even rank $\ge 2$. 
The natural candidates to investigate are the examples of Berridge-Dunwoody (\cref{thm:BD}) which are geometrically realisable by Harlander-Jensen (\cref{thm:HJ}).
The proof of \cref{thm:S+S*-intro} uses the boundary of thickenings construction which, given a finite 2-complex $X$, assigns a closed smooth 4-manifold $M(X)$ given by the boundary of a smooth regular neighbourhood of an embedding $i : X \hookrightarrow \R^5$ (see \cref{ss:doubles}).

\subsection*{Organisation of the paper}

The paper is organised as follows. In \cref{section:unstable}, we will fill in further details on unstable classification, building upon \cref{ss:intro-unstable}. 
Sections \ref{section:RG-modules}-\ref{section:proof-main-algebra} will be devoted to the proof of \cref{thm:main-SF}, Sections \ref{section:module-invariants}-\ref{section:proof-main-topological} to the proof of \cref{thm:main} and \cref{section:proof-main-further} to the proof of \cref{thm:main-SF-further}. 
In \cref{section:induced}, we will prove \cref{thm:main-free-product}. In \cref{section:4-manifolds} we will explore applications to 4-manifolds, culminating in a proof of \cref{thm:S+S*-intro}.
Finally, in \cref{section:problems}, we will propose a number of directions in the study of $P(\Z G)$, $\HT(G)$ and $\M(G)$. 
In particular, we pose Problems \ref{problem:proj-rank=0}-\ref{problem:proj-t-free} concerning projective $\Z G$-modules and Problems \ref{problem:complexes-irred}-\ref{problem:complexes-bound} concerning finite 2-complexes.
We hope that Theorems \ref{thm:simple-modules}, \ref{thm:simple-complexes} and \ref{thm:main-SF-further} go some way towards convincing the reader that progress on these problems is possible.

\subsection*{Acknowledgements}

I would like to thank Martin Dunwoody, Jens Harlander, F. E. A. Johnson and Mark Powell for useful correspondence and a number of helpful comments.
I would also like to thank the referee for many useful suggestions which improved the exposition of this article.
This work was supported by EPSRC grant EP/N509577/1, the Heilbronn Institute for Mathematical Research and a Rankin-Sneddon Research Fellowship from the University of Glasgow.

\section{Preliminaries on $P(\Z G)$, $\HT(G,n)$ and $\M(G)$}
\label{section:unstable}

The aim of this section will be to establish basic properties of the set $P(\Z G)$ of $\Z G$-isomorphism classes of finitely generated non-zero projective $\Z G$-modules, the set $\HT(G,n)$ of homotopy types of finite $(G,n)$-complexes, and the set $\M(G)$ of homeomorphism classes of closed smooth $4$-manifolds $M$ with $\pi_1(M) \cong G$.

Recall from the introduction that, for a set $\mathcal{S}$, a stabilisation is a function $\Sigma : \mathcal{S} \to \mathcal{S}$ such that $\ell(a) := \sup\{k \ge 0: \Sigma^n(a) = \Sigma^{n+k}(b) \text{ for some $b \in \mathcal{S}$, $n \ge 0$}\} < \infty$ where $n, k$ are integers. 
We then have a level function $\ell : \mathcal{S} \to \Z_{\ge 0}$ which is surjective and satisfies $\ell(\Sigma(a)) = \ell(a)+1$ for $a \in \mathcal{S}$. 
We will write $\ell = \ell_{\mathcal{S}}$ when we want to emphasise the set $\mathcal{S}$.
We let $=_{\st}$ denote the corresponding (reduced) stable equivalence relation, i.e. $a =_{\st} b$ if $\Sigma^n(a) = \Sigma^m(b)$ for some $n,m \ge 0$. 

Recall that the rank of a finitely generated projective $\Z G$-module $P$, denoted by $\rank_{\Z G}(P)$, is the rank of the free abelian group $P_G = \Z \otimes_{\Z G} P$. 

\begin{prop} \label{prop:prelim-P(ZG)}
Let $G$ be a group and let $\ell = \ell_{P(\Z G)}$. Then:
\begin{clist}{(i)}
\item
$\Sigma : P(\Z G) \to P(\Z G)$, $P \mapsto P \oplus \Z G$ is a stabilisation, i.e. $\ell(P) < \infty$ for all $P \in P(\Z G)$.
\item
For each $P \in P(\Z G)$, we have $\ell(P) = \rank_{\Z G}(P) - \rank_{\min}([P])$ and, for each $c \in \wt K_0(\Z G)$, 
\[ \rank_{\min}(c) := \min\{\rank_{\Z G}(P_0) : P_0 \in P(\Z G), \, [P_0]=c \in \wt K_0(\Z G) \}\]
is the minimal rank of a projective $\Z G$-module in the class $c$.	
\end{clist}
\end{prop}

\begin{remark}
If $c=0$, then $\rank_{\min}(c)=1$ and so $\ell(P) = \rank_{\Z G}(P)-1$ for $P$ a stably free $\Z G$-module. It is not known whether $\rank_{\min}(c)=1$ for all groups $G$ and all $c \in \wt K_0(\Z G)$ (see \cref{problem:min-rank-in-c}).
\end{remark}

\begin{proof}
(i) Recall that, for $P \in P(\Z G)$, we defined:
\[ \ell(P) := \sup \{ k \ge 0 : P \oplus \Z G^r \cong Q \oplus \Z G^{r+k}, \text{ for some $Q \in P(\Z G)$, $r \ge 0$} \}. \]
If $P \oplus \Z G^r \cong Q \oplus \Z G^{r+k}$, then $\rank_{\Z G}(P) = \rank_{\Z G}(Q) + k$ and so $k \le \rank_{\Z G}(P)$. In particular, we have $\ell(P) \le \rank_{\Z G}(P) < \infty$.

(ii)
Let $P \in P(\Z G)$. Since $\rank_{\Z G}(P \oplus \Z G) = \rank_{\Z G}(P)+1$, we have that $\ell(P) - \rank_{\Z G}(P)$ is invariant under the operation $P \mapsto P \oplus \Z G$ and so is a function of the class $c = [P] \in \wt K_0(\Z G)$. Let $P_0 \in P(\Z G)$ be such that $[P_0]=c$ and $\rank_{\Z G}(P_0) = \rank_{\min}(c)$, where the minimal value exists since it is bounded below by $0$. Since $\rank_{\Z G}(\cdot)$ is minimal at $P_0$, and $\ell(\cdot) - \rank_{\Z G}(\cdot)$ is constant on $c$, it follows that $\ell(\cdot)$ is minimal at $P_0$, i.e. $\ell(P_0)=0$. Hence $\ell(P)-\rank_{\Z G}(P) = -\rank_{\Z G}(P_0)=-\rank_{\min}(c)$.
\end{proof}

For $n \ge 2$ and $X$ a finite $(G,n)$-complex, define the \textit{directed Euler characteristic} $\bs \chi (X) = (-1)^n \chi(X)$.

\begin{prop} \label{prop:prelim-HT(G)}
Let $n \ge 2$, let $G$ be a group of type $F_n$ and let $\ell = \ell_{\HT(G,n)}$. Then:
\begin{clist}{(i)}
\item
$\Sigma : \HT(G,n) \to \HT(G,n)$, $X \mapsto X \vee S^n$ is a stabilisation, i.e. $\ell(X) < \infty$ for all $X \in \HT(G,n)$.
\item
For each $X \in \HT(G,n)$, we have $\ell(X) = \bs \chi(X) - \bs \chi_{\min}(G,n)$ where 
\[ \bs \chi_{\min}(G,n) := \min\{\bs \chi(X) : X \in \HT(G,n)\} \]
is the minimal directed Euler characteristic in $\HT(G,n)$, which always exists.
\end{clist}
\end{prop}

\begin{proof}
(i) Recall that, for $X \in \HT(G,n)$, we defined:
\[ \ell(X) := \sup \{ k \ge 0 : X \vee r S^n \simeq Y \vee (r+k) S^n, \text{ for some $Y \in \HT(G,n)$, $r \ge 0$} \}. \]
If $X \vee r S^n \simeq Y \vee (r+k) S^n$, then $\rank_{\Z}(H_n(X)) = \rank_{\Z}(H_n(Y))+k$ and so $k \le \rank_{\Z}(H_n(X))$.

(ii) This is analogous to the proof of \cref{prop:prelim-P(ZG)} (ii) and so will be omitted for brevity.
\end{proof}

\begin{prop} \label{prop:prelim-M(G)}
Let $G$ be a finitely presented group and let $\ell = \ell_{\M(G)}$. Then:
\begin{clist}{(i)}
\item
$\Sigma : \M(G) \to \M(G)$, $M \mapsto M \# (S^2 \times S^2)$ is a stabilisation, i.e. $\ell(M) < \infty$ for all $M \in \M(G)$.
\item
For each $M \in \M(G)$, we have $\ell(M) = \frac{1}{2}(\chi(M) - \chi_{\min}([M]))$ where, for $c \in \M(G)/\cong_{\st}$, 
\[ \chi_{\min}(c) := \min\{\chi(M) : M \in c\} \]
is the minimal Euler characteristic of a $4$-manifold in $c$, which always exists.
\end{clist}
\end{prop}

\begin{remark}
The same holds with $\M(G)$ replaced by $\M^{\Diff}(G)$.	
\end{remark}

\begin{proof}
(i) Recall that, for $M \in \M(G)$, we defined:
\[ \ell(M) := \sup \{ k \ge 0 : M \# r (S^2 \times S^2) \cong N \# (r+k)(S^2 \times S^2), \text{ for some $N \in \M(G)$, $r \ge 0$} \}. \]
If $M \# r (S^2 \times S^2) \cong N \# (r+k)(S^2 \times S^2)$, then $\rank_{\Z}(H_2(M)) = \rank_{\Z}(H_2(N))+2k$ . This implies that $k \le \frac{1}{2} \rank_{\Z}(H_n(M))$.

(ii) This is analogous to the proof of \cref{prop:prelim-P(ZG)} (ii) and so will be omitted for brevity.
\end{proof}

\section{Preliminaries on $R G$-modules}
\label{section:RG-modules}

Let $G$ be a group, let $R$ be a ring and let $R G$ denote the group ring of $G$ with coefficients in $R$. We will now develop the necessary preliminaries on $R G$-modules.

\subsection{Stably free $R G$-modules}
\label{subsection:stably-finite}

For a ring $R$, a finitely generated (left) $R$-module $S$ is \textit{stably free} if there exists $n$, $m$ such that $S \oplus R^n \cong R^m$.
In order to have a well-defined notion of rank, certain conditions on $R$ must be imposed:
\begin{clist}{(I)}
\item For all $n, m$, $R^n \cong R^m$ implies $n=m$ (\textit{invariant basis number property})
\item For all $n, m$, $S \oplus R^n \cong R^m$ implies $n \le m$ (\textit{surjective rank property})
\item For all $n$, $S \oplus R^n \cong R^n$ implies $S=0$ (\textit{stable finiteness property})
\end{clist}
Suppose $R$ satisfies (I). If $S$ is a stably free $R$-module, then we can define the \textit{rank} of $S$ to be $\rank (S) = m-n$ for any $n$, $m$ such that $S \oplus R^n \cong R^m$. If $R$ satisfies (II), then $\rank(S) \ge 0$ for all $S$. If $R$ satisfies (III), then $S \ne 0$ implies that $\rank(S) \ge 1$.

It is straightforward to see that (III) $\Rightarrow$ (II) $\Rightarrow$ (I). Conversely, examples were given by Cohn \cite{Co66} to show that $(R \ne 0)$ $\not\Rightarrow$ (I) $\not\Rightarrow$ (II) $\not\Rightarrow$ (III). Rings which satisfy (III) are also known as weakly finite and satisfy the equivalent condition that, for all $n$, one-sided inverses in $M_n(R)$ are two-sided, i.e. $uv=1$ if and only if $vu=1$.

We would now like to determine when conditions (I)-(III) hold for $R G$.
The following is a consequence of \cite[Proposition 2.4, Theorem 2.6]{Co66}.

\begin{prop}
Let $R$ be a commutative ring and let $G$ be a group. Then $R G$ has the surjective rank property, and hence also the invariant basis number property.
\end{prop}

It remains to determine when $R G$ is stably finite.
It was shown by Kaplansky \cite{Ka72} that, if $\F$ is a field of characteristic $0$, then $\F G$ is stably finite for all groups $G$. This implies that $\Z G$ is stably finite since $\Z G \subseteq \Q G$.
Kaplansky conjectured that this holds for all fields $\F$, but this remains open.

The best result for general fields $\F$ is the following theorem of Elek-Szab\'{o} \cite{ES04}, which built upon earlier work of Ara, O'Meara and Perera \cite[Theorem 3.4]{AOP02}. 

\begin{thm} \label{thm:ES}
Let $\F$ be a field and let $G$ be a sofic group. Then $\F G$ is stably finite.
\end{thm}

For a definition of sofic, see \cite[p430]{ES04}. For our purposes, it suffices to note that $G=1$ is sofic and that sofic groups are closed under direct/free products, direct/inverse limits, subgroups, and that the extension of an amenable group (see \cite[p227]{AOP02}) by a sofic group is sofic. There is no known example of a non-sofic group.

All groups which will be considered in this article are sofic. We can therefore assume, when needed, that non-trivial stably free $\F G$-modules have rank $\ge 1$.

\subsection{$R G$-modules over free products}
\label{subsection:Bergman}

Fix groups $G_1, \cdots, G_n$, let $G = \ast_{k=1}^n G_k$ denote the free product and let $\iota_k : G_k \hookrightarrow G$ denote the inclusion map for each $k$.

Let $R$ be a ring. If $M_k$ is an $R G_k$-module, then ${\iota_k}_\#(M_k) = R G \otimes_{R G_k}  M_k$ is an $R G$-module. We say that an $R G$-module $M$ is \textit{induced} if there exists $R G_k$-modules $M_k$ and an $R G$-module isomorphism
\[ M \cong {\iota_1}_\#(M_1) \oplus \cdots \oplus {\iota_n}_\#(M_n).\]

We now define two special types of map between induced $R G$-modules.
Firstly, if $M = \bigoplus_{k=1}^n {\iota_k}_\#(M_k)$ and $M' = \bigoplus_{k=1}^n {\iota_k}_\#(M_k')$ are induced $R G$-modules, then an $R G$-module homomorphism $f : M \to M'$ is called an \textit{induced homomorphism} if there exists $R G_k$-module homomorphisms $f_k : M_k \to M_k'$ such that $f = \oplus_{k=1}^n \iota_*(f_k)$.

Now, let $M = \bigoplus_{k=1}^n {\iota_k}_\#(M_k)$ be an induced $R G$-module and suppose there exists $a$ for which $M_a \cong M_a' \oplus R G_a$ for some $R G_a$-module $M_a'$. Then, for any $b \ne a$, there is an isomorphism
\[f_{a,b} : {\iota_a}_\#(M_a' \oplus R G_a) \oplus {\iota_b}_\#(M_b) \to {\iota_a}_\#(M_a') \oplus {\iota_b}_\#(M_b \oplus R G_b)\]
induced by ${\iota_a}_\#(R G_a) \cong R G \cong {\iota_b}_\#(R G_b)$. 
We define a \textit{free transfer isomorphism} on $a,b$ to be the isomorphism $F_{a,b} : M \to M'$ which extends $f_{a,b}$ by the identity map on the other components and where
\[ M' =  {\iota_1}_\#(M_1) \oplus \cdots \oplus {\iota_a}_\#(M_a') \oplus \cdots \oplus {\iota_b}_\#(M_b \oplus R G_b) \oplus \cdots \oplus  {\iota_n}_\#(M_n).\]

The following can be viewed as a special case of Bergman's theorem on modules over coproducts of rings \cite{Be74}. We now restrict to the case where $R = \F$ is a field.

\begin{thm}[Bergman]
\label{thm:Bergman}
Let $M$ be a finitely generated induced $\F G$-module. Then:
\begin{clist}{(i)}
\item If $M' \subseteq M$ is a submodule, then $M'$ is an induced $\F G$-module.
\item If $M'$ is an induced $\F G$-module, then $M \cong M'$ if and only if they are connected by a sequence of induced isomorphisms and free transfer isomorphisms.
\end{clist}
\end{thm}

For the convenience of the reader, we will briefly outline how this can be deduced from Bergman's results. Here will will use the terminology from \cite[p1-4]{Be74}.

\begin{proof}[Proof \normalfont (outline)]
Firstly, note that $\F G$ is a the coproduct of the $\F$-rings $\F G_k$ which are faithful since they come equipped with natural injections $\iota_k : \F G_k \hookrightarrow \F G$. 

Part (i) follows immediately from \cite[Theorem 2.2]{Be74}. For part (ii), suppose $f: M \to M'$ is an isomorphism of $\F G$-modules. By \cite[Theorem 2.3]{Be74}, and the remarks in \cite[p3-4]{Be74}, $f$ is the composition of induced isomorphisms, free transfer isomorphisms and transvections. Since transvections are module automorphisms, omitting them from the composition still leaves an isomorphism of $\F G$-modules. \end{proof}

\begin{corollary} \label{cor:Bergman}
Let $M = \bigoplus_{k=1}^n {\iota_k}_\#(M_k)$ be a finitely generated induced $\F G$-module and suppose each $M_k$ has no direct summand of the form $\F G_k$. Then:
\begin{clist}{(i)}
\item If $M' = \bigoplus_{k=1}^n {\iota_k}_\#(M_k')$ is an induced $\F G$-module, then $M \cong M'$ as $\F G$-modules if and only if $M_k \cong M_k'$ as $\F G_k$-modules for all $k$.
\item $M$ has no direct summand of the form $\F G$.
\end{clist}
\end{corollary}

\begin{proof}
Part (i) follows from \cref{thm:Bergman} (ii) since, if the $M_k$ have no direct summands of the form $\F G_k$, then there are no free transfer isomorphisms by definition. 

To see part (ii) note that, if $M \cong M' \oplus \F G$, then $M' \subseteq M$ is a submodule and so is an induced $\F G$-module by \cref{thm:Bergman} (i). If $M' = \bigoplus_{k=1}^n {\iota_k}_\#(M_k')$, then $M \cong {\iota_1}_\# (M_1' \oplus \F G_1) \oplus \bigoplus_{k=2}^n {\iota_k}_\#(M_k')$ which contradicts the result from (i).
\end{proof}

\subsection{$R G$-modules up to $\Aut(G)$-isomorphism}
\label{subsection:Aut(G)}

If $M$ is an $R G$-module and $\theta \in \Aut(G)$, then we can define $M_\theta$ to be the $R G$-module with the same underlying $R$-module as $M$ but with $G$-action given by $g \cdot_{M_{\theta}} m = \theta(g) \cdot_{M} m$ for $g \in G$ and $m \in M$. We say that $R G$-modules $M$ and $M'$ are \textit{$\Aut(G)$-isomorphic} if $M \cong (M')_\theta$ are isomorphic as $R G$-modules for some $\theta \in \Aut(G)$.
This is equivalent to the existence of a \textit{$\theta$-isomorphism} $f : M \to M'$ which is an $R$-module isomorphism for which $f(g \cdot m) = \theta(g) \cdot f(m)$ for $g \in G$ and $m \in M$.

This has a number of basic properties. In particular, if $M$ and $M'$ are $R G$-modules and $\theta \in \Aut(G)$, then $(M \oplus M')_\theta \cong M_\theta \oplus (M')_\theta$, and $(R G)_\theta \cong R G$ for all $\theta \in \Aut(G)$. 

Recall that a subgroup $N \subseteq G$ is \textit{characteristic} if $\theta(N) = N$ for all $\theta \in \Aut(G)$. We also say that a surjective map $f: G \twoheadrightarrow H$ is characteristic if $\Ker(f) \subseteq G$ is characteristic and, if so, then there is an induced map $\bar{\cdot} : \Aut(G) \to \Aut(H)$.

The following is straightforward (see, for example, \cite[Corollary 7.4]{Ni20a}).

\begin{prop} \label{prop:modules-over-quotients}
Let $G$ be a group, let $f: G \twoheadrightarrow H$ be characteristic and let $\bar{\cdot} : \Aut(G) \to \Aut(H)$ be the map induced by $f$. If $M$ is an $R G$ module and $\theta \in \Aut(G)$, then
$f_\#(M_\theta) \cong (f_\#(M))_{\bar{\theta}}$
are isomorphic as $R H$-modules. 
\end{prop}

The following will be of use in applying \cref{prop:modules-over-quotients} to the case where $G$ is a free product. 
We say that a group $G$ is \textit{indecomposable} if it is non-trivial and $G \cong G_1 \ast G_2$ implies $G_1$ or $G_2$ is trivial.

\begin{prop} \label{prop:free-product-char}
Let $G = G_1 \ast \cdots \ast G_n$ where each $G_k$ is indecomposable and not infinite cyclic. For each $k$, let $f_k : G_k \twoheadrightarrow H_k$ be characteristic and such that, if $G_i \cong G_j$, then $H_i \cong H_j$ and $f_i$, $f_j$ differ by automorphisms of $G_i$, $H_i$.

If $f : G \twoheadrightarrow H_1 \ast \cdots \ast H_n$ is the map with $f \mid_{G_k} = f_k$, then $f$ is characteristic.
\end{prop}

Our proof will be a routine application of the following version of the Kurosh subgroup theorem \cite[Theorem 5.1]{Ma77}.

\begin{thm}[Kurosh subgroup theorem]
Let $G = G_1 \ast \cdots \ast G_n$. If $H \subseteq G$ is a subgroup, then
\[ H = F(X) \ast (\ast_{k=1}^n g_k H_k g_k^{-1})\]
where $F(X)$ is the free group on a set $X$, $H_k \subseteq G_k$ is a subgroup and $g_k \in G$.
\end{thm}

\begin{proof}[Proof of \cref{prop:free-product-char}]
Let $\varphi \in \Aut(G)$. Then $\varphi(G_k) \subseteq G$ is indecomposable and not infinite cyclic and so, by the Kurosh subgroup theorem, we have $\varphi(G_k) = g_{i_k} H_{i_k} g_{i_k}^{-1}$ for some subgroup $H_{i_k} \subseteq G_{i_k}$. Since $\varphi$ is an automorphism, we have:
\[ G = \ast_{k=1}^n (g_{i_k} H_{i_k} g_{i_k}^{-1}) \subseteq \ast_{k=1}^n (g_{i_k} G_{i_k} g_{i_k}^{-1}) \subseteq \ast_{k=1}^n (g_k G_k g_k^{-1}) = G  \]
which implies that $H_{i_k} = G_{i_k}$ and that the $i_k$ are distinct.

Let $N_k = \Ker(f_k) \subseteq G_k$ and note that $N = \Ker(f)$ is generated by the subgroups $g N_k g^{-1}$ for $g \in G$. If $\varphi \in \Aut(G)$, then the above implies that $\varphi \mid_{G_k} = c_{g_{i_k}} \circ \varphi_{i,i_k}$ where $\varphi_{i,i_k} : G_i \to G_{i_k}$ is an isomorphism and $c_{g_{i_k}} : G_{i_k} \to G$ is conjugation by $g_{i_k}$. Since $f_{i}, f_{i_k}$ differ by automorphisms of $G_i, G_{i_k}$, we have $\varphi_{i,i_k}(N_i) = \varphi_{i_k}(N_{i_k})$ for some $\varphi_{i_k} \in \Aut(G_{i_k})$ and so $\varphi_{i,i_k}(N_i) = N_{i_k}$ since $N_{i_k}$ is characteristic. Hence $\varphi(g N_k g^{-1}) = (g g_{i_k}) N_{i_k} (g g_{i_k})^{-1} \subseteq N$ and so $N$ is characteristic.
\end{proof}

\subsection{Duals of $R G$-modules} \label{ss:duals}

We will make use of the following in \cref{section:4-manifolds} on the unstable classification of 4-manifolds. Recall that $R G$ comes equipped with an involution 
\[ \overline{\cdot} : R G \to R G, \quad \sum_{i=1}^r n_i g_i \mapsto \sum_{i=1}^r n_i g_i^{-1}\] 
where $n_i \in R$, $g_i \in G$. This is an anti-isomorphism of rings.

If $M$ is a (left) $R G$-module, then define $M^* = \Hom_{R G}(M,R G)$ to be the (left) $R G$-module with $R G$-action given by letting
\[ (\lambda \varphi) : m \mapsto \varphi(m) \overline{\lambda}  \]
for $\lambda \in R G$ and $\varphi \in M^*$. This satisfies a number of basic properties such as that $(R G^n)^* \cong R G^n$ as $R G$-modules.
The following facts about duals of projective $R G$-modules are standard. Part (ii) says that projective modules are \textit{reflexive}.

\begin{prop} \label{prop:dual-of-proj}
Let $P$ be a finitely generated projective $R G$-module. Then:
\begin{clist}{(i)}
\item
 $P^*$ is a finitely generated projective $R G$-module
\item
The evaluation map $\ev_P : P \to P^{**}$, $x \mapsto (\varphi \mapsto \varphi(x))$ is an isomorphism of $R G$-modules.
\end{clist}
\end{prop}

\begin{proof}
(i) If $P \oplus Q \cong R G^n$, then $P^* \oplus Q^* \cong (R G^n)^* \cong R G^n$.

(ii) If $P \oplus Q \cong R G^n$, then $\ev_{R G^n} = \ev_{P} \oplus \ev_{Q}$ so $\ev_P$ is an isomorphism since $\ev_{R G^n}$ is.
\end{proof}

We will now prove that dualising commutes with the action of $\Aut(G)$ defined in \cref{subsection:Aut(G)}. 

\begin{prop} \label{prop:dual-of-Aut(G)}
Let $M$ be an $R G$-module and let $\theta \in \Aut(G)$. Then there is an isomorphism of $R G$-modules
\[ (M_\theta)^* \cong (M^*)_\theta. \]
In particular, if $R G$-modules $M$ and $N$ are $\Aut(G)$-isomorphic, then $M^*$ and $N^*$ are $\Aut(G)$-isomorphic.
\end{prop}

\begin{proof}
For each $\alpha \in \Aut(G)$, there is an isomorphism of $R G$-modules given by 
\[ \alpha_* : R G \to R G_\alpha, \quad \sum_{i=1}^r n_i g_i \mapsto \sum_{i=1}^r n_i \alpha(g_i) \]
where $n_i \in R$ and $g_i \in G$. Given this, define
\[ f : M^* \to (M_\theta)^*, \quad \varphi \mapsto (\theta^{-1})_* \circ \varphi  \]
where we are viewing $\varphi : M \to R G$ as an $R G$-homomorphism $\varphi : M_\theta \to R G_\theta$. 

We claim that $f$ is a $\theta^{-1}$-isomorphism. This gives the desired result since it implies that $M^* \cong ((M_\theta)^*)_{\theta^{-1}}$ and so $(M^*)_\theta \cong (M_\theta)^*$ by applying $(\,\cdot\,)_\theta$ to both sides.
Firstly, it is clear that this is an $R$-module homomorphism, and is an isomorphism with inverse given by post-composing with $\theta_*$. 
For $g \in G$, $\varphi \in M^*$ and $m \in M$, we have
\begin{align*}
f(g \cdot \varphi)(m) &= (\theta^{-1})_ *(\varphi(m)g^{-1}) = ((\theta^{-1})_* \circ \varphi)(m)(\theta^{-1}(g))^{-1} \\
& = (\theta^{-1}(g) \cdot ((\theta^{-1})_* \circ \varphi))(m) = (\theta^{-1}(g) \cdot f(\varphi))(m)
\end{align*}
and so $f(g \cdot \varphi) = \theta^{-1}(g) \cdot f(\varphi)$, as required.
\end{proof}

\section{Groups of finite cohomological dimension}
\label{section:cd(G)}

We will now recall some basic facts about groups with finite cohomological dimension which are due to Serre \cite{Se71}.
A standard reference is the notes of Bieri \cite{Bi81}.

A group $G$ has \textit{cohomological dimension $n$}, written $\cd(G) = n$, if $n$ is the smallest integer for which there exists a projective resolution of $\Z G$-modules of the form:
\[ 0 \to P_n \to \cdots \to P_1 \to P_0 \to \Z \to 0.\]
This is equivalent to asking that $H^i(G;M)=0$ for all $i > n$ and all $\Z G$-modules $M$ \cite[Proposition 5.1(a)]{Bi81}.
If no such $n$ exists, then we take $\cd(G) = \infty$.

A group $G$ is said to be \textit{of type $\FL$} if, for some $n \ge 0$, there exists a resolution of finitely generated free $\Z G$-modules of the form:
\[ 0 \to F_n \to \cdots \to F_1 \to F_0 \to \Z \to 0\]
The following is \cite[Propositions 1.5, 4.1(b)]{Bi81}.

\begin{prop} \label{prop:FP+cd}
Let $G$ be a group with $\cd(G) = n$. If $G$ is of type $\FL$, then there exists a resolution of finitely generated free $\Z G$-modules of the form:
\[ 0 \to F_n \to \cdots \to F_1 \to F_0 \to \Z \to 0.\]
\end{prop}

We now recall how these conditions are related under amalgamated free products and direct products. The following is \cite[Proposition 2.13(a), Proposition 6.1]{Bi81}.

\begin{lemma} \label{lemma:cd-amalg}
Let $G = G_1 \ast_H G_2$ for groups $G_1$, $G_2$ with a common subgroup $H$.
\begin{clist}{(i)}
\item If $G_1$, $G_2$ are of type $\FL$ and $H$ is of type $\FL$, then $G$ is of type $\FL$
\item If $n = \max\{\cd(G_1),\cd(G_2)\} < \infty$ and  $\cd(H) < n$, then $\cd(G) = n$. 
\end{clist}
\end{lemma}

The following is a consequence of more general results on group extensions which can be found in \cite[Proposition 2.7, Theorem 5.5]{Bi81}.

\begin{lemma} \label{lemma:cd-direct}
Let $G = G_1 \times G_2$ for groups $G_1$, $G_2$.
\begin{clist}{(i)}
\item If $G_1$, $G_2$ are of type $\FL$, then $G$ is of type $\FL$
\item If $\cd(G_1), \cd(G_2) < \infty$, $G_1$ is of type $\FL$ and $H^n(G_1;\Z G_1)$ is $\Z$-free for $n = \cd(G_1)$, then $\cd(G) = \cd(G_1) + \cd(G_2)$.
\end{clist}
\end{lemma}

We will now give a construction of groups which will be the basis for our examples in \cref{thm:main-SF} in the case $d \ge 3$. This is inspired by a construction of Lustig \cite{Lu93}.

\begin{construction} \label{construction}
Let $G$ be a group and let $m \ge 2$ be an integer. Then define
\[ G_+ = ( G \ast \langle r \mid \hspace{-.8mm}-\rangle )/ [r^m,G], \]
which is isomorphic to $(G \times \langle q \mid \hspace{-.8mm}-\rangle) \ast_{\langle q = r^{m} \rangle} \langle r \mid \hspace{-.8mm}-\rangle$. For integers $m_1, \cdots, m_{n-1} \ge 2$, we can define $G_{(n)}$ inductively by letting $G_{(1)} = G$ and $G_{(i+1)} = (G_{(i)})_+$ for $i \ge 1$. We will label the new generator by $r_i$.
The choice of $m_i \ge 2$ will not matter for the purposes of this article; it suffices to consider the case $m_i=2$.
\end{construction}

Let $\iota : G \to G_{(n)}$ be the composition of the natural maps $G_{(i)} \to G_{(i+1)}$ and let $f: G_{(n)} \to G$ be the map which sends $r_i \mapsto 1$ for each $i$. We have that $f \circ \iota = \id_{G}$ and so $\iota$ is injective, $f$ is surjective and $G$ is a retract of $G_{(n)}$.

\begin{prop} \label{prop:group-construction}
Let $n \ge 1$ and let $G$ be a finitely presented group of type $\FL$ with $\cd(G) = d$. Then:
\begin{clist}{(i)}
\item $G_{(n)}$ is a finitely presented group of type $\FL$ with $\cd(G_{(n)})=n+d-1$
\item The map $f: G_{(n)} \twoheadrightarrow G$ is characteristic.
\end{clist}
\end{prop}

In order to prove this, we will first need the following lemma. The proof is identical to the one given in \cite[p174]{Lu93}.

\begin{lemma} \label{lemma:group-construction}
Let $G$ be a torsion free group and let $G_+ = (G \times \langle q \mid \hspace{-.8mm}-\rangle) \ast_{\langle q = r^m \rangle} \langle r \mid \hspace{-.8mm}-\rangle$ for some $m \ge 2$. Then the map $f : G_+ \twoheadrightarrow G$ which sends $r \mapsto 1$ is characteristic.
\end{lemma}

\begin{proof}[Proof of \cref{prop:group-construction}]
It is clear that $G_{(n)}$ is finitely presented. We now prove (i) by induction, noting that it is trivial in the case $n=1$.

Suppose (i) holds for $n$ and note that $G_{(n+1)} \cong (G_{(n)} \times \Z) \ast_{\Z} \Z$. It is well known that $K(\Z,1) \simeq S^1$ and so $\Z$ is of type $\FL$, $\cd(\Z)=1$ and $H^1(\Z;\Z[\Z])=0$. By \cref{lemma:cd-direct}, $G_{(n)} \times \Z$ is of type $\FL$ and $\cd(G_{(n)} \times \Z) = n+d$.  By \cref{lemma:cd-amalg}, this implies that $G_{(n+1)}$ is of type $\FL$ and $\cd(G_{(n+1)})=n+d$ as required.

Since $\cd(G_{(n)}) < \infty$, $G_{(n)}$ is torsion free for all $n$ \cite[Proposition 4.11]{Bi81}. By \cref{lemma:group-construction}, this implies that the map $f_{i+1}: G_{(i+1)} \twoheadrightarrow G_{(i)}$, $r_{i+1} \mapsto 1$ is characteristic for all $i \ge 1$. Hence $f = f_n \circ f_{n-1} \circ \cdots \circ f_{2}$ is characteristic by composition.
\end{proof}

\section{Proof of \cref{thm:main-SF}}
\label{section:proof-main-algebra}

Recall that the trefoil group $T$ is defined as $\pi_1(S^3 \, \setminus \, N(K))$ where $N(K)$ is the knot exterior of the trefoil knot $K \subseteq S^3$. It has presentation $\mathcal{P} = \langle x,y \mid x^2 = y^3 \rangle$.

Let $T''$ denote the second derived subgroup of $T$, i.e. $T'' = (T')'$, and let $f: T \twoheadrightarrow T/T''$ be the quotient map.
Note that $T/T''$ is polycyclic since $(T/T'')' \cong \Z^2$ and $(T/T'')/(T/T'')' \cong \Z$.
The following was shown by P. H. Berridge and M. J. Dunwoody \cite{BD79}, building upon previous work of Dunwoody \cite{Du76}.

\begin{thm}[Berridge-Dunwoody] \label{thm:BD}
There exists infinitely many rank one stably free $\Z T$-modules $S_i$ for $i \ge 1$ such that: 
\begin{clist}{(i)}
\item $S_i \oplus \Z T \cong \Z T^2$	.
\item There exists distinct primes $p_i$ for which $\F_{p_i} \otimes f_\#(S_j) \cong \F_{p_i} [T/T'']$ are isomorphic as $\F_{p_i} [T/T'']$-modules if and only if $i = j$.
\end{clist}
In particular, the $S_i$ are distinct up to $\Z G$-module isomorphism.
\end{thm}

\begin{remark} \label{remark:relation-module}
For $i \ge 0$, let $M_i = \Ker(\cdot\left(\begin{smallmatrix} x^{2i+1}-1 \\ y^{3i+1}-1 \end{smallmatrix}\right) : \Z T^2 \twoheadrightarrow \Z T)$ be the relation module for the generating set $\{x^{2i+1},y^{3i+1}\}$, which is a stably free $\Z T$-module of rank one. It was shown in \cite{BD79} that $\F_p \otimes f_\#(M_i) \cong \F_p[T/T'']$ as $\F_p[T/T'']$-modules if and only if $p \mid i(i+1)$. 
There exists integers $\ell_i$ for $i \ge 1$ and primes $p_i$ such that $p_i \mid \ell_j(\ell_j+1)$ if and only if $i=j$, and so we can take $S_i = M_{\ell_i}$ in \cref{thm:BD}. It is not known whether or not the $M_i$ are all distinct up to $\Z G$-module isomorphism.
\end{remark}

For the rest of this section, fix $k \ge 1$ and $n \ge 1$. Let $G = T_1 \ast \cdots \ast T_k$ where $T_j \cong T$ is the trefoil group and let $G_{(n)}$ be as defined in \cref{construction}.
Since $T$ is a knot group, $T$ has type $\FL$ and $\cd(T) =2$ \cite[p212]{Br82}.
By \cref{lemma:cd-amalg} and \cref{prop:group-construction}, this implies that $G_{(n)}$ has type $\FL$ and $\cd(G_{(n)}) = n+1$. The aim of the rest of this section will be to prove the following theorem which implies \cref{thm:main-SF}.

\begin{thm} \label{thm:main-SF-detailed}
For each $n \ge 1$ and $1 \le m \le k$, there exists infinitely many stably free $\Z G_{(n)}$-modules $\wh S_i$ for $i \ge 1$ such that:
\begin{clist}{(i)}
\item $\wh S_i \oplus \Z G_{(n)} \cong \Z G_{(n)}^{m+1}$.
\item $\wh S_i$ has no direct summand of the form $\Z G_{(n)}$.
\item The $\wh S_i$ for $i \ge 1$ are distinct up to $\Aut(G_{(n)})$-isomorphism of $\Z G_{(n)}$-modules.	
\end{clist}
\end{thm}

Note that the case $m = k$ is sufficient to establish \cref{thm:main-SF}. This result shows that the tree of stably free $\Z G_{(n)}$-modules has branching at all ranks $1 \le m \le k$.
We do not know whether branching occurs at ranks $\ge k+1$, even in the case $G = T$, i.e. when $k=1$.

In order to prove \cref{thm:main-SF-detailed}, we will begin with the following lemma.

\begin{lemma} \label{lemma:f=char}
Let $f_j : T_j \twoheadrightarrow T_j/(T_j)''$ be the quotient maps and let
\[ f: G \twoheadrightarrow (T_1/T_1'') \ast \cdots \ast (T_k/T_k'') \]
be the map induced by the $f_j$. Then $f$ is characteristic. 
\end{lemma}

\begin{proof}
For any group $G$, it is well known that $G' \subseteq G$ is characteristic and so $G'' \subseteq G$ is characteristic also. Hence $f_j$ is characteristic for each $j$. Since $T$ is indecomposable and not infinite cyclic, $f$ is characteristic by \cref{prop:free-product-char}.
\end{proof}

For simplicity, we will begin by proving \cref{thm:main-SF-detailed} in the case $n=1$, i.e. where $G_{(n)}=G$.
From now on, fix $1 \le m \le k$. For integers $i_1, \cdots, i_m$, define
\[S_{i_1, \cdots, i_m} = {\iota_1}_\#(S_{i_1}) \oplus \cdots \oplus {\iota_m}_\#(S_{i_m})    \]
where $\iota_j : T_j \hookrightarrow G$ is the inclusion map.
We will now prove the following as a consequence of Bergman's theorem, which we will apply by using \cref{cor:Bergman}.

\begin{prop} \label{prop:SF-trefoil}
For integers $i_1, \cdots, i_m$, we have:
\begin{clist}{(i)}
\item $S_{i_1, \cdots, i_m} \oplus \Z G \cong \Z G^{m+1}$.
\item $S_{i_1,\cdots, i_m}$ has no direct summand of the form $\Z G$.
\item If $S_{i_1, \cdots, i_m} \cong S_{i_1',\cdots, i_m'}$ are $\Aut(G)$-isomorphic as $\Z G$-modules then, as sets, we have $\{i_1, \cdots, i_m\} = \{i_1',\cdots,i_m'\}$.
\end{clist}
\end{prop}

\begin{proof}
Part (i) is a straightforward consequence of \cref{thm:BD} (i). 

Let $\bar{G} = \ast_{j=1}^n T_j/T_j''$ and let $\text{$\bar{\iota}_j$} : T_j/T_j'' \hookrightarrow \bar{G}$ be inclusion. By \cref{thm:BD} (ii), there exists $p$ such that $\F_p \otimes {f_j}_\#(S_{i_j}) \not \cong \F_p[T_j/T_j'']$ for all $j$. Fix $p$ and note that:
\[ \F_p \otimes f_\#(S_{i_1,\cdots, i_m}) \cong \textstyle \bigoplus_{j=1}^m \F_p \otimes (f \circ \iota_j)_\#(S_{i_j}) \cong \bigoplus_{j=1}^m \text{$\bar{\iota}_j$}_\#(\F_p \otimes {f_j}_\#(S_{i_j})) \]
is an induced $\F_p \bar{G}$ module. 
In order to show that \cref{cor:Bergman} applies, it remains to show that $\F_p \otimes {f_j}_\#(S_{i_j})$ has no direct summand of the form $\F_p[T_j/T_j'']$.

If $\F_p \otimes {f_j}_\#(S_{i_j}) \cong S \oplus \F_p[T_j/T_j'']$, then $S \oplus \F_p[T_j/T_j'']^2 \cong \F_p[T_j/T_j'']^2$.
Since $T_j/T_j''$ is polycyclic, it is amenable and so sofic. By \cref{thm:ES}, $\F_p[T_j/T_j'']$ is stably finite and so $S= 0$. Hence $\F_p \otimes {f_j}_\#(S_{i_j}) \cong \F_p[T_j/T_j'']$, which is a contradiction.

To show (ii) note that, if $S_{i_1,\cdots, i_m}$ has a direct summand $\Z G$, then $\F_p \otimes f_\#(S_{i_1,\cdots, i_m})$ has a direct summand $\F_p \bar{G}$. This contradicts \cref{cor:Bergman} (ii).

To show (iii), suppose that $\{i_1, \cdots, i_m\} \ne \{i_1',\cdots,i_m'\}$ as sets. By symmetry, we can assume that there exists $i_r' \not \in \{i_1, \cdots, i_m\}$. Let $p = p_{i_r'}$ in the notation of \cref{thm:BD}. By the argument above, $\F_p \otimes f_\#(S_{i_1,\cdots, i_m})$ has no direct summand of the form $\F_p \bar{G}$. On the other hand, $\F_p \otimes {f_r}_\#(S_{i_r'}) \cong \F_p[T_r/T_r'']$ which implies that
\begin{align*} \F_p \otimes f_\#(S_{i_1',\cdots, i_m'}) & \cong \textstyle \bigoplus_{j=1, j \ne r}^m \text{$\bar{\iota}_j$}_\#(\F_p \otimes {f_j}_\#(S_{i_j})) \oplus \F_p \bar{G} \\
& \cong \textstyle \bigoplus_{j=1, j \ne r}^{m-1} \text{$\bar{\iota}_j$}_\#(\F_p \otimes {f_j}_\#(S_{i_j})) \oplus \F_p \bar{G}^2 \cong \cdots \cong \F_p \bar{G}^m.	
\end{align*}
If $S_{i_1,\cdots, i_m} \cong S_{i_1',\cdots, i_m'}$ are $\Aut(G)$-isomorphic, then $S_{i_1,\cdots, i_m} \cong (S_{i_1',\cdots, i_m'})_\theta$ for some $\theta \in \Aut(G)$. By \cref{lemma:f=char}, $f$ is characteristic and so, by \cref{prop:modules-over-quotients}, $f_\#((S_{i_1',\cdots, i_m'})_\theta) \cong (f_\#(S_{i_1',\cdots, i_m'}))_{\bar{\theta}}$ for some $\bar{\theta} \in \Aut(\bar{G})$. In particular, we have:
\[ \F_p \otimes f_\#(S_{i_1,\cdots, i_m}) \cong (\F_p \otimes f_\#(S_{i_1',\cdots, i_m'f}))_{\bar{\theta}} \cong (\F_p \bar G^m)_{\bar{\theta}} \cong \F_p \bar{G}^m\]
which is a contradiction.
\end{proof}

\begin{proof}[Proof of \cref{thm:main-SF-detailed}]
Let $\iota : G \hookrightarrow G_{(n)}$ and $f: G_{(n)} \twoheadrightarrow G$ be as defined in \cref{section:cd(G)}. This satisfies $f \circ \iota = \id_G$ and, by \cref{prop:group-construction}, $f$ is characteristic. Define $\wh S_i = \iota_\#(\wh S_{i_1, \cdots, i_m})$, where $i_j = i$ for all $j$. By \cref{prop:SF-trefoil}, it is now straightforward to check that the $\wh S_i$ has the required properties.
\end{proof}

We conclude this section with extended remarks on \cref{thm:main-SF} and  \cref{thm:main-SF-detailed}.

\sss{Relation modules}

By \cref{remark:relation-module}, $S_i$ is the relation module for the generating set $\{x^{2\ell_i+1},y^{3\ell_i+1}\}$ of $T$. It follows that $S_{i_1, \cdots, i_m}$ is the relation module for the generating set $\{x_i^{2\ell_i+1},y_i^{3\ell_i+1}\}_{i=1}^k$ of $G = T_1 \ast \cdots \ast T_k$ where $T_i = \langle x_i, y_i \mid x_i^2 = y_i^3 \rangle$.

\sss{Change of field}

In the proof of \cref{prop:SF-trefoil}, the $\Z G$-modules were distinguished by passing to $\F_p G$ for various $p$. An alternate approach is to instead pass to $\Q G$ and use the results of Lewin \cite{Le82}.
If $\F = \F_p$ or $\Q$, then one can show:

\begin{thm} \label{thm:COF-SF}
Let $k \ge 1$ and let $G = T_1 \ast \cdots \ast T_k$. Then there exists a stably free $\Z G$-module $S$ of rank $k$ such that $S \otimes \F$ is a non-free stably free $\F G$-module.	
\end{thm}

It would be interesting to know if one could detect infinitely many distinct stably free $\Z G$-modules of rank $k$ on $\F G$ for some $\F = \F_p$ or $\Q$, even in the case $k=1$.

\sss{Alternate constructions}

There are more ways to deduce \cref{thm:main-SF} in the case $d \ge 3$ from the case $d=2$.
By \cref{prop:SF-trefoil} and the proof of \cref{thm:main-SF-detailed}, it suffices to find a finitely presented group $G$ with $\cd(G) = d$ and a characteristic quotient $f : G \twoheadrightarrow \ast_{i=1}^N T$ for some $N \ge k$.
Two such constructions are as follows.
\begin{clist}{(1)}
\item 
Let $G = \ast_{i=1}^r (\ast_{j=1}^{n_i} T)_{(d-1)}$ where $1 \le n_1 \le \cdots \le n_r$ and $N = \sum_{i=1}^r n_i$. Then $\cd(G) = d$ and there is a characteristic quotient $f : G \twoheadrightarrow \ast_{i=1}^N T$.
For example, we can take $G = (\ast_{i=1}^k T)_{(d-1)}$ as above, or $G = \ast_{i=1}^k T_{(d-1)}$ (see \cref{thm:main-SF-further-detailed}).

\item
Let $G = (\ast_{i=1}^N T) \times \Gamma$ where $\Gamma$ is a finitely presented group with $\cd(\Gamma) = d-2$, $Z(\Gamma)=1$ and which does not contain $\ast_{i=1}^N T$ as a direct factor. By \cref{lemma:cd-direct}, we have $\cd(G) = d$. If $N \ge 2$, then $Z(\ast_{i=1}^N T)=1$ and it can be deduced from \cite[Corollary 2.2]{Jo83} that $f: G \twoheadrightarrow \ast_{i=1}^N T$ is characteristic. 

For example: If $d=3$, let $\Gamma$ be a free group of rank $\ge 2$.
If $d = 4$, let $\Gamma$ be a surface group of genus $\ge 2$.
If $d \ge 5$, let $\Gamma \subseteq L$ be a cocompact torsion free lattice in a non-compact simple Lie group $L$ with dimension $d-2$ over its maximal compact subgroup. 
Note that there are infinitely many such $\Gamma$ up to commensurability.
I am indebted to F. E. A. Johnson for this observation.
\end{clist}

\section{Module invariants of CW-complexes}
\label{section:module-invariants}

Let $X$ be a CW-complex and recall that its cellular chain complex $C_*(\wt X)$ is a chain complex of free $\Z[\pi_1(X)]$-modules under the monodromy action. The chain homotopy type of $C_*(\wt X)$ is a homotopy invariant for $X$ and so, for all $n$, the $\Z[\pi_1(X)]$-module $H_n(C_*(\wt X))$ is also a homotopy invariant.

If $G$ is a group and $\rho: \pi_1(X) \cong G$, then every $\Z[\pi_1(X)]$-module $M$ can be converted to a $\Z G$-module with action $g \cdot_{\Z G} m := \rho^{-1}(g) \cdot_{\Z[\pi_1(X)]} m$ for $g \in G$ and $m \in M$. In this notation, $H_n(C_*(\wt X))_\rho$ is a $\Z G$-module. We will denote this by $H_n(X;\Z G)$ when $\rho$ is understood. If $\rho' : \pi_1(X) \cong G$ and $\theta = \rho \circ (\rho')^{-1} \in \Aut(G)$, then $H_n(C_*(\wt X))_{\rho'} \cong (H_n(C_*(\wt X))_\rho)_\theta$. In particular, the $\Aut(G)$-isomorphism class of $H_n(X;\Z G)$ is a homotopy invariant and is independent of the choice of $\rho$.

The aim of this section will be to consider how $H_n(X;\Z G)$ changes under wedge product. We will also give a mild variation of this invariant under group quotients.

\subsection{Homology of a wedge product}
\label{subsection:wedge}

The following is presumably well-known. However, we were not able to locate a suitable reference in the literature.

\begin{prop} \label{prop:CW-of-a-wedge}
Let $X_1$, $X_2$ be CW-complexes with a single 0-cell such that $\pi_1(X_k) \cong G_k$. Let $X = X_1 \vee X_2$ which has $\pi_1(X) \cong G$ where $G= G_1 \ast G_2$. Then:
\[ 
C_i(\wt X) = 
\begin{cases}
{\iota_1}_\#(C_i(\wt X_1)) \oplus {\iota_2}_\#(C_i(\wt X_2)), & \text{if $i \ge 1$} \\
\Z G , & \text{if $i=0$}
\end{cases}
\]
where $\partial_i = {\iota_1}_\#(\partial^{X_1}_{i}) \oplus {\iota_2}_\#(\partial^{X_2}_{i})$ for $i \ge 2$, $\partial_1 = ({\iota_1}_\#(\partial^{X_1}_{1}), {\iota_1}_\#(\partial^{X_2}_{1}))$ and $\partial_0 = \varepsilon_G$.
\end{prop}

\begin{proof}
It suffices to compute an explicit model for $\wt X$ in terms of $\wt X_1$ and $\wt X_2$. Such a model, which is often attributed to Scott-Wall \cite{SW79}, is provided by taking the graph of spaces structure on $X = X_1 \vee X_2$ and lifting it to $\wt X$. 

Define a graph $(V,E)$ with vertex set $V = V(X_1) \sqcup V(X_2)$ where $V(X_1)$ is the set of elements in $G_1 \ast G_2$ with final term in $G_2$, i.e. the identity $e$ as well as the elements of the form $g_n \cdots g_1 g_1$ for $n \ge 1$ where $g_i \in G_2 \setminus \{ 1 \}$ when $i$ is odd and $g_i \in G_1 \setminus \{ 1 \}$ otherwise. Define $V(X_2)$ similarly. Note that, whilst $V(X_1) \cap V(X_2) = \{1\}$ as subsets of $G_1 \ast G_2$, the elements $1 \in V(X_i)$ are not identified in $V$.

Define $E = \bigsqcup_{v \in V(X_1)} (G_1 \, \setminus \, \{1\})_v \sqcup \bigsqcup_{v \in V(X_2)} (G_2 \, \setminus \, \{1\})_v \sqcup \{e_{1,1} \}$ where, for each $v \in V(X_1)$ and $g \in G_1 \, \setminus \, \{1\}$, we have a directed edge $e_{v,vg} = (g)_v$ from $v$ to $vg$ which is labeled by $g \in G$. Similarly for $V(X_2)$ and $G_2$. The edge $e_{1,1}$ from $1 \in V(X_1)$ to $1 \in V(X_2)$ is labeled by $1 \in G$.

Let $\ast \in X_i$ denote the $0$-cell and, for each $g \in G_i$, let $\ast_g \in \wt X_i$ denote its corresponding lift.
Our model is the CW-complex
\[ X_{(V,E)} = \left(\bigsqcup_{v \in V(X_1)} (\wt X_1)_v \sqcup \bigsqcup_{v \in V(X_2)} (\wt X_2)_v \right)/\sim \]
where, if we have a directed edge $e_{v_1,v_2} \in E$ with label $g \in G$, then $(\ast_g)_{v_1} \sim (\ast_1)_{v_2}$ where, if $v_1 \in V(X_i)$, then $(\ast_g)_{v_1} \in (\wt X_i)_{v_1}$ and similarly for $(\ast_1)_{v_2}$. 
By comparing with the construction in \cite{SW79}, we have $\wt X \simeq X_{(V,E)}$.

We now determine the induced action of $G = G_1 \ast G_2$ on $X_{(V,E)}$. Note that $G_1$ acts $(\wt X_1)_1$ by monodromy and freely permutes the $\ast_g \in (\wt X_1)_1$. This action extends to all of $X_{(V,E)}$ inductively, and similarly for the action of $G_2$ on $(\wt X_2)_2$. Since $G = \langle G_1, G_2 \rangle$, this determines the full action of $G$ on $X_{(V,E)}$. 

It now remains to read off the cell structure of $X_{(V,E)}$ under this $G$-action. For $i \ge 1$, the $i$-cells lie in the interior of the copies of $\wt X_1$, $\wt X_2$ and so are unaffected by the relation $\sim$. This implies that:
\[C_i(X_{(V,E)}) = \bigoplus_{v \in V(X_1)} v \cdot C_i(\wt X_1) \oplus \bigoplus_{v \in V(X_2)} v \cdot C_i(\wt X_2) \]
as an abelian group. Since $G$ acts on the $V(X_j)$ in the natural way, and the elements of $V(X_j)$ are coset representatives for $G/G_j$, we have that: 
\[ \bigoplus_{v \in V(X_1)} v \cdot C_i(\wt X_1) \cong \Z G \otimes_{\Z G_j} C_i(\wt X_j) \cong {\iota_j}_\#(C_i(\wt X_j)) \]
as $\Z G$-modules. We can determine $C_0(X_{(V,E)})$ and the $\partial_i$ similarly.
\end{proof}

\begin{corollary} \label{cor:pi_2-of-wedge}
Let $X_1$ and $X_2$ be CW-complexes with a single $0$-cell such that $\pi_1(X_i) \cong G_i$. Let $X = X_1 \vee X_2$ which has $\pi_1(X) \cong G$ where $G= G_1 \ast G_2$. Then:
\[ H_n(X;\Z G) \cong {\iota_1}_\#(H_n(X; \Z G_1)) \oplus {\iota_2}_\#(H_n(X;\Z G_2)). \]
\end{corollary}

\begin{remark}
This could be deduced from the Mayer-Vietoris sequence for homology with local coefficients \cite[Theorem 2.4]{Wh78}, though the above argument is more direct.
\end{remark}

\subsection{Homology under group quotients} 
\label{subsection:change-of-group}

Let $X$ be a CW-complex with $\rho : \pi_1(X) \cong G$ and let  $C_*(\wt X)_\rho$ be the corresponding chain complex of $\Z G$-modules. 
If $f: G \twoheadrightarrow H$ is a quotient of groups, then $f_\#(C_*(\wt X)_\rho)$ is a chain complex of free $\Z H$-modules with boundary maps $\id_{\Z H} \otimes \partial_i$, and $H_n(f_\#(C_*(\wt X)_\rho))$ is a $\Z H$-module. 
We will denote this by $H_n(X;\Z H)$ when $f$ and $\rho$ are understood.

Subject to conditions on $f$, this give an additional homotopy invariant for $X$.

\begin{prop} \label{prop:ZH-homology}
If $f$ is characteristic, then the $\Aut(H)$-isomorphism class of $H_n(X;\Z H)$ is a homotopy invariant and is independent of the choice of $\rho$.	
\end{prop}

\begin{proof}
If $C_*(\wt X)_\rho \simeq C_*(\wt Y)_{\rho'}$ are chain homotopic as chain complexes of $\Z G$-modules, then $f_\#(C_*(\wt X)_\rho) \simeq f_\#(C_*(\wt Y)_{\rho'})$ are chain homotopic as chain complexes of $\Z H$-modules.
Let $\theta \in \Aut(G)$. Since $f$ is characteristic, \cref{prop:modules-over-quotients} implies that $f_\#((C_*(\wt X)_\rho)_\theta) \cong (f_\#(C_*(\wt X)_\rho))_{\bar{\theta}}$
for some $\bar{\theta} \in \Aut(H)$. The result now follows.
\end{proof}

\section{Algebraic classification of finite $(G,n)$-complexes}
\label{section:Gn-complexes}

A \textit{$(G,n)$-complex} is an $n$-dimensional CW-complex $X$ such that $\pi_1(X) \cong G$ and the universal cover $\wt X$ is $(n-1)$-connected. By contracting a maximal spanning tree, $X$ is homotopy equivalent to a $(G,n)$-complex with a single $0$-cell. For convenience, we will now assume that a $(G,n)$-complex has a single $0$-cell which is the basepoint.

If $i \ge 2$, then $\pi_i(X) \cong \pi_i(\wt X)$ as abelian group. In this way, we can view $\pi_i(X)$ as a $\Z G$-module under the monodromy action.
If $2 \le i < n$, then $\pi_i(X) = 0$ since $\wt X$ is $(n-1)$-connected. If $i = n$, then the Hurewicz theorem implies that:
\[ \pi_n(X) \cong H_n(\wt X ;\Z) \cong H_n(X;\Z G) \]
as $\Z G$-modules. In particular, \cref{cor:pi_2-of-wedge} applies to $\pi_n(X)$.

\subsection{Algebraic $n$-complexes and the D2 problem}
\label{subsection:algebraic-n-complexes}

Let $G$ be a group. 
An \textit{algebraic $n$-complex over $\Z G$} is an exact chain complex:
\[ E = (F_n \xrightarrow[]{\partial_n} \cdots \xrightarrow[]{\partial_2} F_1 \xrightarrow[]{\partial_1} F_0 \xrightarrow[]{\partial_0} \Z \to 0)\]
where the $F_i$ are finitely generated stably free $\Z G$-modules. 

Let $\Alg(G,n)$ denote the equivalence classes of algebraic $n$-complexes over $\Z G$ up to chain homotopy equivalences of the unaugmented complex $(F_i,\partial_i)_{i=1}^n$.
The \textit{$n$th homotopy group} of $E$ is the $\Z G$ module $\pi_n(E) = \Ker(\partial_n)$ and is an invariant of the chain homotopy class of $E$. 
If $n \ge 2$, we can assume the $F_i$ are free since every algebraic $n$-complex is chain homotopy equivalent to such a complex.

 Let $\PHT(G,n)$ denote the polarised homotopy types of finite $(G,n)$-complexes, i.e. the homotopy types of pairs $(X,\rho)$ where $\rho: \pi_1(X) \cong G$.
If $(X,\rho) \in \PHT(G,n)$, then $C_*(\wt X)_{\rho}$ is a chain complex of $\Z G$-modules such that $H_0(C_*(\wt X)_{\rho}) \cong \Z$ and $H_i(C_*(\wt X)_{\rho}) = 0$ for $1 \le i < n$. In particular, there is a map:
\[ \Psi: \PHT(G,n) \to \Alg(G,n). \]

Recall that a finitely presented group $G$ has the \textit{D2 property} if every finite CW-complex $X$ such that $\pi_1(X) \cong G$, $H_i(\wt X;\Z) = 0$ for $i > 2$ and $H^3(X;M)=0$ for all finitely generated $\Z G$-modules $M$ is homotopy equivalent to a finite 2-complex.
The following is a mild improvement of Wall's results on finiteness conditions for CW-complexes due to Johnson \cite{Jo03a} and Mannan \cite{Ma09}. 
This precise version follows from \cite[Corollary 8.27]{Jo12a} in the case $n \ge 3$ and \cite[Theorem 2.1]{Ni19} in the case $n=2$.

\begin{prop} \label{prop:realisation-thm}
Let $G$ be a finitely presented group. 
If $n \ge 3$, then $\Psi$ is bijective.
If $n = 2$, then $\Psi$ is injective and is bijective if and only if $G$ has the {\normalfont D2} property.
\end{prop}

\begin{remark}
The first part is often vacuous since there are finitely presented groups $G$ for which no algebraic $n$-complex over $\Z G$ exists for all $n \ge 3$. The first example was found by Stallings in \cite{St63} (see also \cite[Proposition 2.14]{Bi81}) and was later generalised to a class of right-angled Artin groups by Bestvina-Brady \cite[Main Theorem]{BB97}.
\end{remark}

\subsection{Realising $\Z G$-modules by algebraic $n$-complexes}
\label{subsection:pi_n-realisation}

The $n$th \textit{stable syzygy} $\Omega_{n}^G(\Z)$ is the set of $\Z G$-modules $M$ for which $M \oplus \Z G^i \cong \pi_{n-1}(E) \oplus \Z G^j$ for some $i, j \ge 0$ and some algebraic $(n-1)$-complex $E$ over $\Z G$. We will denote this by $\Omega_n(\Z)$ when the choice of $G$ is clear from the context. This is well-defined and does not depend on the choice of $E$ \cite[Theorem 8.9]{Jo12a}.
It also comes with a map:
\[ \pi_n : \Alg(G,n) \to \Omega_{n+1}(\Z). \]

The following can be found in \cite[Proposition 8.18]{Jo12a}. Note that a group $G$ being finitely presented and of type FL is equivalent to the group $G$ admitting a finite Eilenberg-Maclane space $K(G,1)$.

\begin{prop} \label{prop:realisation-of-syzygies}
Let $n \ge 2$ and let $G$ be an infinite finitely presented group of type $\FL$ such that $H^{n+1}(G;\Z G)=0$. Then $\pi_n$ is bijective.
\end{prop}

The following is a straightforward consequence of Propositions \ref{prop:FP+cd} and \ref{prop:realisation-of-syzygies}.

\begin{prop} \label{prop:syzygies-cd-finite}
Let $G$ be a finitely presented group of type $\FL$ with $\cd(G) = d$.
\begin{clist}{(i)}
\item If $n \ge d$, then $\Omega_{n}(\Z)$ is the set of stably free $\Z G$-modules
\item If $n \ge d$, then $\pi_n: \Alg(G,n) \to \Omega_{n+1}(\Z)$ is bijective
\item If $n = d-1$, then $0 \not \in \IM(\pi_{n} : \Alg(G,n) \to \Omega_{n+1}(\Z))$.
\end{clist}
\end{prop}

\begin{remark}
This implies that, for $n \ge 2$, $\pi_n : \Alg(G,n) \to \Omega_{n+1}(\Z)$ is not surjective whenever $\cd(G) = n+1$ (for example, $G = \Z^{n+1}$). This was noted in \cite[p107]{Jo12a}.
\end{remark}

It is possible to see that $\cd(G) \le n$ implies that $\pi_n: \Alg(G,n) \to \Omega_{n+1}(\Z)$ is surjective directly (see, for example, \cite[Theorem 4]{HJ06b}).
The following is now clear.

\begin{corollary} \label{cor:homotopy-cd-finite}
Let $n \ge 3$ and let $G$ be a finitely presented group of type $\FL$ with $\cd(G) = n$. Then $\pi_n$ gives a one-to-one correspondence between homotopy types of finite $(G,n)$-complexes and $\Aut(G)$-isomorphism classes of stably free $\Z G$-modules.
\end{corollary}

Finally, we note the following where $\rank(P)$ denotes the stably free rank of $P$.

\begin{prop} \label{prop:rank-computation}
Let $G$ be a finitely presented group of type $\FL$ with $\cd(G) = d$ and let $n \ge d-1$. Then $\chi(X) = k + \chi_{\min}(G,n)$ if and only if:
\[ \rank(\pi_n(X)) = k + \min\{ \rank(\pi_n(X_0)) : \text{$X_0$ a finite $(G,n)$-complex} \}. \]
In particular, if $n \ge \max\{3, d\}$, then $k = \rank(\pi_n(X))$. 
\end{prop}

\section{Proof of \cref{thm:main}}
\label{section:proof-main-topological}

We will now prove \cref{thm:main} separately in the two cases of non-minimal Euler characteristic ($k \ge 1$) and minimal Euler characteristic ($k =0$). 
Throughout, $T_i \cong T$ will denote the trefoil group and, for a group $G$, the group $G_{(n)}$ will be as defined in \cref{construction}.

\subsection{Finite $(G,n)$-complexes with non-minimal Euler characteristic}
\label{subsection:non-min-EC}

The aim of this section will be to prove the following. Note that, in the case $n \ge 3$, we could also take $G$ to be one of the other groups listed at the end of \cref{section:proof-main-algebra}.

\begin{thm} \label{thm:main-non-min-EC}
Let $n \ge 2$, let $k \ge 1$ and let $G = (T_1 \ast \cdots \ast T_k)_{(n-1)}$ as in \cref{construction}. Then, for all $1 \le m \le k$, there exists infinitely many finite $(G,n)$-complexes $\wh X_i$ such that:
\begin{clist}{(i)}
\item $\pi_n(\wh X_i) \cong \wh S_i$ as $\Z G$-modules (where $\wh S_i$ is as defined in \cref{thm:main-SF-detailed})
\item $\chi(\wh X_i) = m + \chi_{\min}(G,n)$
\item $\wh X_i \not \simeq Y_i \vee S^2$ for any finite $(G,n)$-complex $Y_i$.
\end{clist}
\end{thm}

Since the $\Aut(G)$-isomorphism class of $\pi_n(\wh X_i)$ is a homotopy invariant, it follows that the $\wh X_i$ are homotopically distinct by \cref{thm:main-SF-detailed}. By restricting to the case $m=k$, this implies \cref{thm:main} for $k \ge 1$. 

We will begin with the case $n=2$, where $G = T_1 \ast \cdots \ast T_k$. Let $S_i$ be the stably free $\Z T$-modules from \cref{thm:BD} and, for $1 \le m \le k$, recall that:
\[ S_{i_1, \cdots, i_m} = {\iota_1}_\#(S_{i_1}) \oplus \cdots \oplus {\iota_m}_\#(S_{i_m}).\]
The case of interest will be $\wh S_i = S_{i_1, \cdots, i_m}$ where $i_j = i$ for all $j$.

The main result which we will use is the following, which is \cite[Theorem 4.5]{HJ06a}.

\begin{thm}[Harlander-Jensen] \label{thm:HJ}
The trefoil group $T$ has presentations
\[ \mathcal{P}_i = \langle x,y,a,b \mid x^2=y^3, a^2=b^3, x^{2i+1}=a^{2i+1}, y^{3i+1}=b^{3i+1} \rangle \]
for $i \ge 0$. For each $i$, there exists $\ell_i$ such that $S_i \cong \pi_2(X_{\mathcal{P}_{\ell_i}})$ as $\Z T$-modules.
\end{thm}

\begin{remark}
Note that $\mathcal{P}_0 \simeq \langle x, y \mid x^2=y^3, 1 \rangle$ and $\mathcal{P}_1$ is homotopy equivalent to the presentation found by Dunwoody in \cite{Du76}.	
\end{remark}

Let $X_i = X_{\mathcal{P}_{\ell_i}}$ for each $i \ge 1$. For integers $i_j \ge 1$, define:
\[ X_{i_1, \cdots, i_n} = X_{i_1} \vee \cdots \vee X_{i_n}\]
which is a finite $2$-complex with $\pi_1(X_{i_1, \cdots, i_n}) \cong T_1 \ast \cdots \ast T_k$. Let $\wh X_i = X_{i_1, \cdots, i_n}$ where $i_j = i$ for all $j$.
By repeated application of \cref{cor:pi_2-of-wedge}, we have that $\pi_2(X_{i_1,\cdots,i_n}) \cong S_{i_1, \cdots, i_n}$ and so $\pi_2(\wh X_i) \cong \wh S_i$. 
Since $\rank(\wh S_i) = m$, we have that $\chi(\wh X_i) = m + \chi_{\min}(G)$ by \cref{prop:rank-computation}. Finally, if $\wh X_i \simeq Y_i \vee S^2$, then:
\[ \wh S_i \cong \pi_2(\wh X_i) \cong \pi_2(Y_i) \oplus (\Z G \otimes_{\Z} \pi_2(S^2)) \cong \pi_2(Y_i) \oplus \Z G \]
which is a contradiction since $\wh S_i$ has no summand of the form $\Z G$ by \cref{thm:main-SF-detailed}.	
This completes the proof of \cref{thm:main-non-min-EC} in the case $n=2$.

We will now consider the case $n \ge 3$. where $G = (T_1 \ast \cdots \ast T_k)_{(n-1)}$. By \cref{thm:main-SF-detailed}, there exists stably free $\Z G$-modules $\wh S_i$ of rank $m$ and which have no summand of the form $\Z G$. By \cref{prop:group-construction}, we have that $\cd(G) = n$ and so, by \cref{cor:homotopy-cd-finite}, there exists finite $(G,n)$-complexes $\wh X_i$ such that $\pi_n(\wh X_i) \cong \wh S_i$. We can now argue similarly to the case $n=2$. This completes the proof of \cref{thm:main-non-min-EC}.

\sss{Application to Syzygies}

We now discuss consequences of \cref{thm:main-non-min-EC} for syzygies.
Recall that a $\Z G$-module $M_0 \in \Omega_n(\Z)$ is \textit{minimal} if $M \in \Omega_n(\Z)$ implies that $M \oplus \Z G^i \cong M_0 \oplus \Z G^j$ for some $i \le j$. For $k \ge 0$, we say that $M \in \Omega_n(\Z)$ has \textit{level $k$} if $M \oplus \Z G^i \cong M_0 \oplus \Z G^j$ where $j-i = k$ and $M_0$ is minimal.
If $X$ is a finite $(G,n)$-complex, then $\pi_n(X) \in \Omega_{n+1}(\Z)$. If $\cd(G) = n$ and $\pi_n(X)$ is stably free $\Z G$-module of rank $k$, then $\pi_n(X)$ has level $k$. 
Hence, by \cref{thm:main-non-min-EC}, we have: 

\begin{corollary} \label{cor:syzygies}
For all $n \ge 3$ and $k \ge 1$, there exists a group $G$ and infinitely many $\Z G$-modules $M_i \in \Omega_n(\Z)$ at level $k$ which are distinct up to $\Aut(G)$-isomorphism.
\end{corollary}

\sss{$\F$-homotopy}

For a field $\F$, an \textit{$\F$-homotopy equivalence} is a map $f: X \to Y$ such that $\pi_1(f)$ is a group isomorphism and $\pi_i(f) \otimes \F : H_i(\wt X) \otimes \F \to H_i(\wt Y) \otimes \F$ is bijective for $i \ge 2$. If $\F = \F_p$ or $\Q$ then, similarly to \cref{thm:COF-SF}, we can show:

\begin{thm} \label{thm:COF-main}
If $n \ge 2$, $k \ge 1$ and $G = (T_1 \ast \cdots \ast T_k)_{(n-1)}$, then there exists $\F$-homotopically distinct finite $(G,n)$-complexes $X_1$, $X_2$ with $\chi(X_i) = k + \chi_{\min}(G,n)$.
\end{thm}

\subsection{Finite $(G,n)$-complexes with minimal Euler characteristic}
\label{subsection:min-EC}

The following is the main result of \cite{Lu93}. Note that $T_{(n)}$ denotes the result of applying \cref{construction} to the trefoil group $T$.

\begin{thm}[Lustig] \label{thm:lustig}
Let $G = T_{(2)}$. Then there exists infinitely many homotopically distinct finite $2$-complexes $X_i$ for $i \ge 1$ such that $\pi_1(X_i) \cong G$ and $\chi(X_i) = 1$.
\end{thm}

The aim of this section will be to give the following generalisation of this result:

\begin{thm} \label{thm:main-min-EC}
Let $n \ge 2$ and let $G = T_{(n)}$. Then there exists infinitely many finite $(G,n)$-complexes $\wh X_i$ for $i \ge 1$ such that:
\begin{clist}{(i)}
\item $H_n(\wh X_i; \Z T) \cong S_i$ as $\Z T$-modules (where $S_i$ is as defined in \cref{thm:BD})
\item $\chi(\wh X_i) = \chi_{\min}(G,n)$
\end{clist}
\end{thm}

\begin{remark}
This corrects a statement made in \cite[Section 5]{HJ06b} where it was suggested that, in the case $n=2$ where the complexes $\wh X_i$ coincide with the $X_i$ from \cref{thm:lustig}, we have $\chi(\wh X_i) = 1 + \chi_{\min}(T_{(2)})$. 
In fact, for each $n \ge 2$, we have $\chi(\wh X_i) = \chi_{\min}(T_{(n)},n) = 1-n$.
\end{remark}

By \cref{prop:ZH-homology}, the $\Aut(T)$-isomorphism class of $H_n(\wh X_i; \Z T)$ is a homotopy invariant and so the $\wh X_i$ are homotopically distinct by \cref{thm:BD}. Hence this implies \cref{thm:main} in the case $k = 0$.

We will begin with the following lemma, which can be verified directly. For a group $G$, the group $G_+$ will be as defined in \cref{construction}.

\begin{lemma} \label{lemma:alg(n)-alg(n+1)}
Let $n \ge 2$, let $G$ be a group and let $E = (\Z G^{d_i}, \partial_i)_{i=0}^n \in \Alg(G,n)$ be an algebraic $n$-complex over $\Z G$ such that $d_0=1$.
If $G_{+} = (G \times \langle q \mid \hspace{-.8mm}-\rangle) \ast_{\langle q = r^2 \rangle} \langle r \mid \hspace{-.8mm}-\rangle$, then:
\[ E_+ = (\Z G_+^{d_{n+1}+d_n} \xrightarrow[]{\partial_{n+1}^+} \Z G_+^{d_n+d_{n-1}} \xrightarrow[]{\partial_{n}^+} \cdots \xrightarrow[]{\partial_2^+} \Z G_+^{d_1+d_0} \xrightarrow[]{\partial_1^+} \Z G_+^{d_0} \xrightarrow[]{\varepsilon_{G_+}} \Z \to 0) \in \Alg(G_+,n+1)\]
where $\partial_1^+ = \cdot\left(\begin{smallmatrix}\partial_1 \\ r-1 \end{smallmatrix}\right)$, $\partial_2^+ = \cdot\left(\begin{smallmatrix}\partial_2 & 0 \\ r^2-1 & -\partial_1 \cdot (r+1) \end{smallmatrix}\right)$ and $\partial_i^+ = \cdot\left(\begin{smallmatrix}\partial_i & 0 \\ r^2-1 & -\partial_{i-1} \end{smallmatrix}\right)$ for $i \ge 3$. The $\partial_*$ are the induced maps and we take $d_{n+1}=0$, $\partial_{n+1}=0$. Since $d_0=1$, $\varepsilon_{G_+}$ is just the augmentation map.
\end{lemma}

\begin{remark}
A similar statement can be made when $G_{+} = (G \times \langle q \mid \hspace{-.8mm}-\rangle) \ast_{\langle q = r^m \rangle} \langle r \mid \hspace{-.8mm}-\rangle$ for $m \ge 2$. In this case, $r^2-1$ should be replaced by $r^m-1$ in the definition of $\partial_i^+$ for $i \ge 2$, and $-\partial_1 \cdot (r+1)$ should be replaced by $-\partial_1 \cdot (r^{m-1}+\cdots+r+1)$ in the definition of $\partial_2^+$.
\end{remark}

Let $\mathcal{P} = \langle x, y \mid x^2 = y^3 \rangle$ be the standard presentation for $T$ and note that:
\[ C_*(\wt X_{\mathcal{P}}) \cong (\Z T \xrightarrow[]{\partial_2} \Z T^2 \xrightarrow[]{\partial_1} \Z T \xrightarrow[]{\varepsilon_T} \Z \to 0) \in \Alg(T,2)\]
where $\partial_1 = \cdot\left(\begin{smallmatrix} x-1 \\ y-1 \end{smallmatrix}\right)$ and $\partial_2 = \cdot\left(\begin{smallmatrix} x+1 & -(y^2+y+1) \end{smallmatrix}\right)$.
This has $\pi_2(C_*(\wt X_{\mathcal{P}})) = 0$.

For each $n \ge 1$, define $\wt E_n \in \Alg(T_{(n)},n+1)$ by $\wt E_1 = C_*(\wt X_{\mathcal{P}})$ and $\wt E_{n} = (\wt E_{n-1})_+$ for $n \ge 2$ using  \cref{lemma:alg(n)-alg(n+1)}.
Let $E_n \in \Alg(T_{(n)},n)$ denote the restriction to the first $n+1$ terms in $\wt E_n$. 
Note that $\pi_{n+1}(\wt E_n) = 0$ and so $\pi_{n}(E_n) = \IM(\partial_{n+1}^{\wt E_n}) \cong \Z T_{(n)}$.

For $n \ge 2$, let $\Delta_n = \partial_{n}^{\wt E_n}$ denote the final boundary map in $E_n$, so that:
\[ 
\Delta_1 = \partial_1 \cdot (r_1+1), \quad 
\Delta_{n} = \cdot\left(\begin{smallmatrix} v_{n}  & 0  \\ r_{n-1}^2-1 & -\Delta_{n-1} \end{smallmatrix}\right) : \Z T_{(n)}^{n+1} \to \Z T_{(n)}^{\frac{n(n+1)}{2}}
\]
where $v_{n} = (r_{n-2}^2-1, (-1)(r_{n-3}^2-1), \cdots, (-1)^{n-3}(r_1^2-1), (-1)^{n-2}\partial_2)$. Here $\Delta_1$ is defined for the purposes of this definition and 
does not coincide with $\partial_1^{E_1} = \partial_1$.

Let $\alpha_n, \beta_n$ denote the last two row vectors in $\Delta_n$, which are defined by: 
\[ \alpha_1 = (x-1)(r_1+1), \quad \beta_1 = (y-1)(r_1+1)\]
\[ \alpha_n = (\underbrace{0, \cdots, 0}_{n-2} , r_{n-1}^2-1, 0, -\alpha_{n-1}), \quad \beta_n = (\underbrace{0, \cdots, 0}_{n-1} , r_{n-1}^2-1, -\beta_{n-1}).\]
For each $i \ge 0$, let $\alpha_n^{(i)} = \Sigma_x \alpha_n$, $\beta_n^{(i)} = \Sigma_y \beta_n$ where $\Sigma_x = \sum_{j=0}^{2i} x^j$, $\Sigma_y = \sum_{j=0}^{3i} y^j$. 

We will now show that following, where we adopt the notation of \cref{subsection:change-of-group}.

\begin{prop} \label{prop:alg-alterations}
For $n \ge 2$, let $\Delta_n^{(i)}$ be the matrix $\Delta_n$ but with $\alpha_n, \beta_n$ replaced by $\alpha_n^{(i)}, \beta_n^{(i)}$, and let $E_{n}^{(i)}$ be the resolution $E_n$ but with $\Delta_n$ replaced by $\Delta_n^{(i)}$. Then:
\begin{clist}{(i)}
\item $E_n^{(i)} \in \Alg(T_{(n)},n)$
\item If $f : T_{(n)} \twoheadrightarrow T$, then $H_n(E_n^{(i)};\Z T) \cong \Ker(\cdot\left(\begin{smallmatrix}x^{2i+1}-1 \\ y^{3i+1}-1 \end{smallmatrix}\right))$ as $\Z T$-modules.
\end{clist}
\end{prop}

For the convenience of the reader, we will write this explicitly in the case $n=2$:

\newcommand{\matrixa}{\cdot\left(\begin{smallmatrix}
x+1 & -(y^2+y+1) & 0 \\ 
(r_1^2-1)\Sigma_x & 0 & (1-x^{2i+1})(r_1+1) \\ 
0 & (r_1^2-1)\Sigma_y & (1-y^{3i+1})(r_1+1)  
\end{smallmatrix}\right)}

\newcommand{\columna}{\cdot\left(\begin{smallmatrix}x-1 \\ y-1 \\ r_1 -1 \end{smallmatrix}\right)}

\[ E_2^{(i)} = (\Z T_{(2)}^3 \xrightarrow[]{\matrixa} \Z T_{(2)}^3 \xrightarrow[]{\columna} \Z T_{(2)} \xrightarrow[]{\varepsilon} \Z \to 0). \]

In order to prove this, we will first need the following technical lemma.

\begin{lemma} \label{lemma:identity}
Let $G$ be a group with $T \subseteq G$. For $i = 1, 2, 3$, there exists $\lambda_i, \mu_i \in \Z T \subseteq \Z G$ such that, for all $r \in G$, we have:
\[(r-1,0,1-x) = \lambda_1 \cdot (\Sigma_x(r-1),0,1-x^{2i+1}) + \lambda_2 \cdot (0,\Sigma_y(r-1),1-y^{3i+1}) + \lambda_3 \cdot (\partial_2,0) \cdot (r-1)\]
\[(0,r-1,1-y) = \mu_1 \cdot (\Sigma_x(r-1),0,1-x^{2i+1})
+ \mu_2 \cdot (0,\Sigma_y(r-1),1-y^{3i+1}) + \mu_3 \cdot (\partial_2,0) \cdot (r-1)\]
\end{lemma}

\begin{proof}
Note that $x = (x^{2i+1})^{1-6i} (y^{3i+1})^{6i}$, $y = (y^{3i+1})^{6i+1}(x^{2i+1})^{-6i}$, which implies that $\langle x, y \rangle = \langle x^{2i+1}, y^{3i+1} \rangle$ and so $\Z T \cdot \{ x-1, y-1 \} = \Z T \cdot \{ x^{2i+1}-1,y^{3i+1}-1 \}$. 
We can use this to show that:
\begin{align*}
\lambda_1 &= -(x^{2i+1})^{1-6i}(1+x^{2i+1}+\cdots+(x^{2i+1})^{6i-2})\\ 
\lambda_2 &= (x^{2i+1})^{1-6i}(1+y^{3i+1}+\cdots+(y^{3i+1})^{6i-1})
\end{align*}
are such that $x-1 = \lambda_1(x^{2i+1}-1)+\lambda_2(y^{3i+1}-1)$. By comparing with the other terms in $(r-1,0,x-1)$, we can then show that
\begin{align*}
\lambda_3 &= (x^{2i+1})^{1-6i}(1+x^2+\cdots+(x^2)^{6i^2+2i-1}) \hspace{6mm}
\end{align*}
has the desired property. The $\mu_i$ can be found similarly and are omitted.
\end{proof} 

\begin{proof}[Proof of \cref{prop:alg-alterations}]
To prove (i), it suffices to show that $\IM(\cdot\Delta_n^{(i)}) = \IM(\cdot\Delta_n)$ for $i \ge 1$. We have $\IM(\cdot\Delta_n^{(i)}) \subseteq \IM(\cdot\Delta_n)$, so it remains to show $\alpha_n, \beta_n \in \IM(\cdot\Delta_n^{(i)})$.

By the proof of \cref{lemma:identity}, we have $\Z T \cdot \{ x-1, y-1 \} = \Z T \cdot \{ x^{2i+1}-1,y^{3i+1}-1 \}$. It follows that $\Z T \cdot \{ \alpha_1, \alpha_2\} = \Z T \cdot \{\alpha_1^{(i)}, \beta_1^{(i)}\}$ which implies that $\alpha_1, \beta_1 \in \IM(\cdot\Delta_1^{(i)})$.
The case $n = 2$ is done in \cref{lemma:identity}, which provides $\lambda_i$ such that:
\[ \alpha_2 = \lambda_1 \cdot \alpha_2^{(i)} + \lambda_2 \cdot \beta_2^{(i)} + \lambda_3(r_1^2-1) \cdot (\partial_2,0)  \]
and similarly for $\mu_i$ and $\beta_2$.
Let $\gamma_1, \cdots , \gamma_{n-1}$ denote the first $n-1$ rows of $\Delta_n$, the remaining two rows being $\alpha_n, \beta_n$. It is now straightforward to see that:
\[ \alpha_n = \lambda_1 \cdot \alpha_n^{(i)} + \lambda_2 \cdot \beta_n^{(i)} + \lambda_3(-1)^n((r_{n-1}^2-1) \cdot \gamma_1 + \sum_{i=2}^{n-1} (-1)^{i}(r_{n-i}^2-1) \cdot \gamma_{i}) \]
for $n \ge 2$, and similarly for $\beta_n$. Hence $\alpha_n, \beta_n \in \IM(\cdot\Delta_n^{(i)})$ for all $n \ge 2$.

To prove (ii), note that $H_n(E_n^{(i)}; \Z T) = \Ker(f_\#(\Delta_n^{(i)}))$. For each $n \ge 2$, we have:
\[ f_\#(\Delta_n^{(i)}) = \cdot \left(\begin{smallmatrix} f_\#(v_n) & 0 \\ 0 & - f_\#(\Delta_{n-1}^{(i)})  \end{smallmatrix}\right). \]
Since $f_\#(v_n) = (0, \cdots, 0, (-1)^{n-2}\partial_2)$ is injective, this implies that $\Ker(f_\#(\Delta_n^{(i)})) = \Ker(-f_\#(\Delta_{n-1}^{(i)}))$ and so, by induction:
\[ \Ker(f_\#(\Delta_n^{(i)})) \cong \Ker(f_\#(\Delta_1^{(i)})) = \Ker(\cdot\left(\begin{smallmatrix} 2(x^{2i+1}-1) \\ 2(y^{3i+1}-1) \end{smallmatrix}\right)) = \Ker(\cdot\left(\begin{smallmatrix} x^{2i+1}-1 \\ y^{3i+1}-1 \end{smallmatrix}\right)). \qedhere \]
\end{proof}
\vspace{-1mm}

Let $G = T_{(n)}$. For each $i \ge 1$, there exists $\ell_i$ such that $\Ker(\cdot\left(\begin{smallmatrix} x^{2\ell_i+1}-1 \\ y^{3\ell_i+1}-1 \end{smallmatrix}\right)) \cong S_i$ where the $S_i$ are as defined in the discussion following \cref{thm:BD}

If $n \ge 3$, then \cref{prop:realisation-thm} implies that there exists finite $(G,n)$-complexes $\wh X_i$ such that $C_*(X) \simeq E_n^{(\ell_i)}$ are chain homotopy equivalent where $X$ is the universal cover of $\wh X_i$. This is also true when $n=2$ by taking $\wh X_i = X_i = \mathcal{P}_{\ell_i}$ where:
\[ \mathcal{P}_i = \langle a, b, c \mid a^2=b^3, [a^2,b^{2i+1}], [a^2,c^{3i+1}] \rangle \]
are the presentations given by Lustig in \cite{Lu93}.

By \cref{prop:alg-alterations}, $H_n(\wh X_i; \Z T) \cong S_i$ as $\Z T$-modules. It is straightforward to see that 
\[\rank(\pi_n(E_n^{(\ell_i)})) = \rank(\pi_n(E_n)) = 1.\] 
By \cref{prop:modules-over-quotients}, $\cd(G) = n+1$ and so $0 \not \in \IM(\pi_n : \PHT(G,n) \to \Omega_{n+1}(\Z))$ by \cref{prop:syzygies-cd-finite}. Hence, by \cref{prop:rank-computation}, we have $\chi(X_i) = \chi_{\min}(G,n)$. This completes the proof of \cref{thm:main-min-EC}. By combining with \cref{thm:main-non-min-EC}, this completes the proof of \cref{thm:main}.

\section{Proof of Theorem \ref{thm:main-SF-further}}
\label{section:proof-main-further}

The aim of this section will be to prove the following theorem which its Theorems \ref{thm:main-SF-further}. We will also give an application of this to the construction of non-finite $(G,n)$-complexes.
The proofs are similar to that of Theorems \ref{thm:main-SF} and \ref{thm:main} and so many of the details will be omitted.

We will let $T$ denote the trefoil group and $T_{(d-1)}$ be the result of applying \cref{construction}.

\begin{thm} \label{thm:main-SF-further-detailed}
Let $d \ge 2$ and let $G = \ast_{i=1}^\infty T_{(d-1)}$. Then $\cd(G) = d$ and, for all $k \ge 1$, there exists infinitely many stably free $\Z G$-modules of rank $k$ which are distinct up to $\Aut(G)$-isomorphism.
\end{thm}

Let $S_i$ denote the stably free $\Z T$-modules of \cref{thm:BD} and let $\iota_j : T_j \hookrightarrow G$.

\begin{proof}
Let $k \ge 1$ and let $\wh S_i^{(k)} = \bigoplus_{j=1}^k {\iota_j}_\#(S_i)$ for $i \ge 1$. Since $\wh S_i^{(k)} \oplus \Z G \cong \Z G^{k+1}$, the $\wh S_i^{(k)}$ are stably free $\Z G$-modules of rank $k$.
Let $f: G \twoheadrightarrow \ast_{j=1}^\infty T_j/T_j''$ be induced by the characteristic quotients $f_j : (T_j)_{(d-1)} \twoheadrightarrow T_j/T_j''$.
This is characteristic by a mild generalisation of \cref{prop:free-product-char} which applies since $T_j$ is finitely generated.

For $p$ prime, we have that $\F_p \otimes f_\#(\wh S_i^{(k)}) \cong \oplus_{j=1}^k \text{$\bar{\iota}_j$}_\#(\F_p \otimes {f_j}_\#(S_i))$ where $\text{$\bar{\iota}_j$} : T_j/T_j'' \hookrightarrow \ast_{j=1}^\infty T_j/T_j''$ is the inclusion map.
Similarly to the proof of \cref{thm:main-SF-detailed}, there exists primes $p_i$ for $i \ge 1$ such that $\F_{p_i} \otimes {f_j}_\#(S_i) \cong \F_{p_i} [T_j/T_j'']$ if and only if $i = j$. Since \cref{thm:Bergman} and \cref{cor:Bergman} also holds for infinite free products (see \cite{Be74}), we get the $\F_p \otimes f_\#(\wh S_i^{(k)})$ are distinct up to $\Aut(G)$-isomorphism. Since $f$ is characteristic, the $\wh S_i^{(k)}$ are distinct up to $\Aut(G)$-isomorphism also.
\end{proof}

\begin{thm} \label{thm:main-further-detailed}
Let $n \ge 2$ and let $G = \ast_{i=1}^\infty T_{(n-1)}$. 
Then there exists an aspherical $(G,n)$-complex $Y$ such that, for all $k \ge 1$, there are infinitely many homotopically distinct $(G,n)$-complexes $X_i$ with $X_i \vee S^n \simeq Y \vee (k+1)S^n$.
\end{thm}

\begin{proof}
By \cref{lemma:alg(n)-alg(n+1)}, there exists $\wt E_{n-1} \in \Alg(T_{(n-1)},n)$ with $\pi_n(\wt E_{n-1}) = 0$.  If $n \ge 3$, then \cref{prop:realisation-thm} implies that there exists a finite $(G,n)$-complex $Y_0$ such that $C_*(\wt Y_0) \simeq \wt E_{n-1}$ are chain homotopy equivalent. This is also true when $n=2$ by taking $Y_0 = X_{\mathcal{P}}$ where $\mathcal{P} = \langle x, y \mid x^2=y^3 \rangle$ is the standard presentation for $T$. Hence, for all $n \ge 2$, $Y = \vee_{i=1}^\infty Y_0$ is an aspherical $(G,n)$-complex.

For all $i \ge 1$, let $X_i = \bigvee_{j=1}^k \wh X_i \vee \bigvee_{j=k+1}^\infty Y_0$ where the $\wh X_i$ are the finite $(T_{(n-1)},n)$-complexes such that $\pi_n(\wh X_i) \cong S_i$ which were constructed in \cref{thm:main-non-min-EC}.
Then $X_i$ is a $(G,n)$-complex such that:
\[ \pi_n(X_i) \cong \bigoplus_{j=1}^k {\iota_j}_\#(\pi_n(\wh X_i)) \oplus \bigoplus_{j=k+1}^\infty {\iota_j}_\#(\pi_n(Y)) \cong \bigoplus_{j=1}^k {\iota_j}_\#(S_i) = \wh S_i^{(k)}.  \]
Since the $\wh S_i^{(k)}$ are distinct up to $\Aut(G)$-isomorphism, this implies that the $X_i$ are homotopically distinct.
By \cref{thm:BD} and \cref{prop:syzygies-cd-finite}, we have that $\wh X_i \vee S^n \simeq Y_0 \vee 2S^n$. It follows that $X_i \vee S^n \simeq Y \vee (k+1)S^n$, as required.
\end{proof}

\section{Some remarks on induced module decompositions}
\label{section:induced}

Recall that Theorems \ref{thm:main-SF} and \ref{thm:main} concerned stably free $\Z G$-modules and finite $2$-complexes $X$ with $\pi_1(X) \cong G$ where $G = \ast_{i=1}^k G_i$.
In our example, $\pi_2(X) \otimes \F_p$ was an induced $\F_p G$-module whose component $\F_p T$-modules $M_i$ were unique up to $\F_p T$-isomorphism where $G_i = T$ is the trefoil group.

The aim of this section will be to investigate the extent to which this applies to all groups of the form $G = \ast_{i=1}^k G_i$ and to $\pi_2(X)$ rather than just $\pi_2(X) \otimes \F$.
For simplicity, we will restrict to the case of $2$-complexes. However, all results have analogues for $(G,n)$-complexes for $n \ge 3$.
The main result is \cref{thm:main-free-product}, which was stated in the introduction. Part (i) will be proven as \cref{thm:non-existence} and part (ii) will be proven as \cref{thm:non-uniqueness}.

\subsection{Existence of induced module decompositions}
\label{subsection:existence}

We will begin by considering the question of existence. From now on, we will take $\F$ to be a field.

\begin{prop}[Existence over \text{$\F[G_1 \ast \cdots \ast G_k]$}] \label{prop:existence}
Let $X$ be a finite $2$-complex with $\pi_1(X) \cong G_1 \ast \cdots \ast G_k$. Then $\pi_2(X) \otimes \F$ is an induced $\F[G_1 \ast \cdots \ast G_k]$-module.
\end{prop}

\begin{proof}
Let $X_i$ be a finite $2$-complex with $\pi_1(X_i) \cong G_i$. Then $\pi_1(\vee_{i=1}^k X_i) \cong \ast_{i=1}^k G_i$ and so there exists $a, b \ge 0$ such that $X \vee aS^2 \simeq \vee_{i=1}^k X_i \vee bS^2$. This implies that
\[ (\pi_2(X) \otimes \F) \oplus \F G^a \cong {\iota_1}_\#((\pi_2(X_1) \otimes \F) \oplus \F G_1^b) \oplus  \bigoplus_{j = 2}^k {\iota_j}_\#(\pi_2(X_j) \otimes \F) \]
and so $\pi_2(X) \oplus \F$ is a submodule of an induced $\F[G_1 \ast \cdots \ast G_k]$-module. Hence, by \cref{thm:Bergman}, $\pi_2(X) \oplus \F$ is an induced $\F[G_1 \ast \cdots \ast G_k]$-module.
\end{proof}

\begin{thm}[Non-existence over \text{$\Z[G_1 \ast \cdots \ast G_k]$}] \label{thm:non-existence}
For all $k \ge 2$, there exists a finite $2$-complex $X$ with $\pi_1(X) \cong G_1 \ast \cdots \ast G_k$ such that $\pi_2(X)$ is not an induced $\Z[G_1 \ast \cdots \ast G_k]$-module.
\end{thm}

In order to prove this, we will need the following method of proving that presentation complexes are homotopy equivalent.
If $\mathcal{P} = \langle x_1, \dots, x_n \mid r_1, \dots, r_m \rangle$, then an \textit{elementary transformation} on $\mathcal{P}$ is an operation that replaces a relator $r_i$ with:
\begin{clist}{(i)}
\item $\omega r_i \omega^{-1}$ for a word $\omega \in F(x_1 \cdots, x_n)$ (\textit{conjugation})
\item $r_i^{-1}$ (\textit{inversion})
\item $r_i r_j$ or $r_j r_i$ for some $j \ne i$ (\textit{left or right multiplication}).
\end{clist}
We say that two group presentations $\mathcal{P}$ and $\mathcal{Q}$ are \textit{$Q$-equivalent} if they are related by a sequence of elementary transformations. If $\mathcal{P}$ and $\mathcal{Q}$ are $Q$-equivalent, then $X_{\mathcal{P}}$ and $X_{\mathcal{Q}}$ are (simple) homotopy equivalent \cite[p20-29]{HMS93}.

We begin by noting the following, which is a generalisation of \cite[Theorem 3]{HLM85}.

\begin{prop} \label{prop:def(G*H)}
Let $k \ge 1$ and let $m_i, n_i \ge 1$ for $i =1, \cdots, k$. Suppose there exists integers $r_i$, $q_i$ such that $(q_i,q_j)=1$ for all $i \ne j$ and, for all $i$, we have: 
\[ r_i^{m_i}-1=n_iq_i, \qquad r_i \equiv 1 \mod n_i , \qquad (m_i,n_i) \ne 1. \]
Then $G = \ast_{i=1}^k (\Z/m_i \times \Z/n_i)$ has a presentation
\[ \mathcal{P} = \langle a_1, b_1, \dots, a_k, b_k \mid a_1^{m_1}, \dots, a_k^{m_k}, a_1b_1a_1^{-1}b_1^{-r_1}, \dots, a_kb_ka_k^{-1}b_k^{-r_k}, b_1^{n_1} \cdots b_k^{n_k} \rangle  \]
of deficiency $-1$. Furthermore, if $\mathcal{P}_i = \langle a,b \mid a^{n_i}, b^{m_i}, [a,b] \rangle$ is the standard presentation for $\Z/m_i \times \Z/n_i$, then $X_{\mathcal{P}} \vee (k-1)S^2 \simeq X_{\mathcal{P}_1} \vee \cdots \vee X_{\mathcal{P}_k}$.
\end{prop}

The conditions on $m_i, n_i$ are satisfied in the case where $m_i = n_i = p_i$ for distinct primes $p_i$. In particular, this applies to all groups of the form $G = \ast_{i=1}^k (\Z/p_i)^2$.

\begin{proof}
That proof that $\mathcal{P}$ presents $G$ is similar to the case $k=2$ (see \cite[Theorem 3]{HLM85}), as so will be omitted.
Let $\mathcal{P}_+$ denote the presentation $\mathcal{P}$ with additional relations $b_1^{n_1}, \cdots, b_{k-1}^{n_{k-1}}$, so that $X_{\mathcal{P}_+} \simeq X_{\mathcal{P}} \vee (k-1)S^2$. In order to show that $X_{\mathcal{P}} \vee (k-1)S^2 \simeq X_{\mathcal{P}_1} \vee \cdots \vee X_{\mathcal{P}_k}$, it therefore suffices to show that $\mathcal{P}_+$ and $\mathcal{P}_1 \ast \cdots \ast \mathcal{P}_k$ are $Q$-equivalent. To see this, note that we can replace $b_1^{n_1} \cdots b_k^{n_k} \leadsto b_k^{n_k}$ by left-multiplying by the $b_i^{-n_k}$ for $1 \le i < k$. Since $r_i \equiv 1 \mod n_i$, we can then replace $a_ib_ia_i^{-1}b_i^{-r_i} \leadsto [a_i,b_i]$ by successively right-multiplying by $b_i^{n_i}$.
\end{proof}

We say that two $\Z G$-modules $M$ and $M'$ are \textit{stably isomorphic}, written $M \cong_s M'$, if there exists $a,b \ge 0$ such that $M \oplus \Z G^a \cong M' \oplus \Z G^b$.

\begin{lemma} \label{lemma:stable-uniqueness}
For $1 \le i \le k$, let $M_i$, $M_i'$ be finitely generated $\Z G_i$-lattices such that 
\[ {\iota_1}_\#(M_1) \oplus \cdots \oplus {\iota_k}_\#(M_k) \cong {\iota_1}_\#(M_1') \oplus \cdots \oplus {\iota_k}_\#(M_k')\] 
as $\Z[G_1 \ast \cdots \ast G_k]$-modules. Then $M_i \cong_s M_i'$ for all $1 \le i \le k$.
\end{lemma}

\begin{proof}
For $1 \le i \le k$, let $q_i : G_1 \ast \cdots \ast G_k \twoheadrightarrow G_i$ be the projection map. By applying $(q_1)_\#$ to the given isomorphism of $\Z[G_1 \ast G_2]$-modules, we get that
\[ M_1 \oplus \bigoplus_{j=2}^k {(q_1 \circ \iota_j)}_\#(M_j) \cong M_1' \oplus \bigoplus_{j=2}^k {(q_1 \circ \iota_j)}_\#(M_j') \]
as $\Z G_1$-modules.
If $j \ne 1$, then $q_1 \circ \iota_j : G_j \to G_1$, $g \mapsto 1$. If $M$ is a finitely generated $\Z G_j$-module, then ${(q_1 \circ \iota_j)}_\#(M) \cong \Z G_1 \otimes_{\Z} (\Z \otimes_{\Z G_j} M)$. If $\Z \otimes_{\Z G_j} M \cong \Z^{r_M} \oplus F_M$ for $F_M$ a finite abelian group and $r_M \ge 0$, then ${(q_1 \circ \iota_j)}_\#(M) \cong \Z G_1^{r_M} \oplus F_M G_1$.

In particular, for some finite abelian groups $F, F'$ and some $r,r' \ge 0$, we have $M_1 \oplus \Z G_1^{r} \oplus F G_1 \cong M_1' \oplus \Z G_1^{r'} \oplus F' G_1$.
Since $M_1, M_1'$ are $\Z G_1$-lattices, this $\Z G_1$-isomorphism must induce isomorphisms $F G_1 \cong F' G_2$ and $M_1 \oplus \Z G_1^{r} \cong M_1' \oplus \Z G_1^{r'}$. Hence $M_1 \cong_s M_1'$ and, by symmetry, we have that $M_i \cong_s M_i'$ for all $1 \le i \le k$.
\end{proof}

\begin{proof}[Proof of \cref{thm:non-existence}]
Let $p_1, \cdots, p_k$ be distinct primes and let $G = \ast_{i=1}^k (\Z/p_i)^2$. By 	\cref{prop:def(G*H)}, $G$ has a presentation $\mathcal{P}$ of deficiency $-1$. 
We claim that $\pi_2(X_{\mathcal{P}})$ is not an induced $\Z[G_1 \ast \cdots \ast G_k]$-module, where $G_i = (\Z/p_i)^2$ for all $i$.

Suppose that $\pi_2(X_{\mathcal{P}}) = {\iota_1}_\#(M_1) \oplus \cdots \oplus {\iota_k}_\#(M_k)$ for $\Z G_i$-modules $M_i$. 
Again by \cref{prop:def(G*H)}, we have that $X_{\mathcal{P}} \vee (k-1)S^2 \simeq X_{\mathcal{P}_1} \vee \cdots \vee X_{\mathcal{P}_k}$ where the $\mathcal{P}_i = \langle a,b \mid a^{p_i}, b^{p_i}, [a,b] \rangle$ are the standard presentations for $G_i$.
Hence, we have:
\[ {\iota_1}_\#(M_1 \oplus \Z G_1^{k-1}) \oplus  \bigoplus_{j = 2}^k {\iota_j}_\#(M_j) \cong \bigoplus_{j = 1}^k {\iota_j}_\#(\pi_2(X_{\mathcal{P}_j})). \]
By \cref{lemma:stable-uniqueness}, this implies that $M_i \cong_s \pi_2(X_{\mathcal{P}_i})$ for all $i$ and so $M_i \in \Omega_3^{G_i}(\Z)$.
	
It follows from \cite[Proposition 2.1]{Sw65}	 that $\pi_2(X_{\mathcal{P}_i}) \in \Omega_3^{G_i}(\Z)$ is minimal and so $M_i \oplus \Z G_i^{r_i} \cong \pi_2(X_{\mathcal{P}_i}) \oplus \Z G_i^{s_i}$ for some integers $r_i \le s_i$. This gives that:
\[ \pi_2(X_{\mathcal{P}}) \oplus \Z G^{s_1 + \cdots + s_k + k -1} \cong \pi_2(X_{\mathcal{P}}) \oplus \Z G^{r_1 + \cdots + r_k}. \]
By \cite[Proposition 2.1]{Jo12a}, $\sum s_i + k -1 = \sum r_i \le \sum s_i$ which is a contradiction.	
\end{proof}

\subsection{Uniqueness of induced module decompositions}
\label{subsection:uniqueness}

We will now turn to the question of uniqueness. The following is an immediate consequence of \cref{cor:Bergman}.	

\begin{prop}[Uniqueness over \text{$\F[G_1 \ast \cdots \ast G_k]$}] \label{prop:uniqueness}
Let $X$ be a finite $2$-complex with $\pi_1(X) \cong G_1 \ast \cdots \ast G_k$.
If $\pi_2(X) \otimes \F \cong {\iota_1}_\#(M_1) \oplus \cdots \oplus {\iota_k}_\#(M_k)$ for $\F G_i$-modules $M_i$ such that $\F G_i \nmid M_i$ 
, then the $M_i$ are unique up to $\F G_i$-module isomorphism.
\end{prop}

\begin{thm}[Non-uniqueness over \text{$\Z[G_1 \ast \cdots \ast G_k]$}] \label{thm:non-uniqueness}
For all $k \ge 2$, there exists finite $2$-complexes $X_i$, $Y_i$ with $\pi_1(X_i) \cong \pi_1(Y_i) \cong G_i$ for $1 \le i \le k$ such that 
\[ \pi_2(X_1 \vee \cdots \vee X_k) \cong \pi_2(Y_1 \vee \cdots \vee Y_k) \] 
but, for all $i$, $\Z G_i \nmid \pi_2(X_i), \, \pi_2(Y_i)$ and $\pi_2(X_i) \not \cong \pi_2(Y_i)$ are not $\Aut(G_i)$-isomorphic.
\end{thm}

\begin{remark}
Note that this implies \cref{thm:main-free-product} (ii) since it implies that:
\[ \pi_2(X_1 \vee \cdots \vee X_k) \cong  {\iota_1}_\#(\pi_2(X_1)) \oplus \cdots \oplus {\iota_k}_\#(\pi_2(X_k)) \cong {\iota_1}_\#(\pi_2(Y_1)) \oplus \cdots \oplus {\iota_k}_\#(\pi_2(Y_k)).\]	
\end{remark}

In order to prove this, we will begin by proving the following.
We note that this holds for a larger class of abelian groups than elementary abelian $p$-groups.

\begin{prop} \label{prop:abelian-collapse}
Let $k \ge 2$ and let $p_i$ be distinct primes and $n_i \ge 1$ for $i =1, \cdots, k$. If $\mathcal{P}_i$, $\mathcal{P}_i'$ are two presentations for $G_i = (\Z/p_i)^{n_i}$ with the same deficiency, then $X_{\mathcal{P}_1} \vee \cdots \vee X_{\mathcal{P}_k} \simeq X_{\mathcal{P}_1'} \vee \cdots \vee X_{\mathcal{P}_k'}$.
\end{prop}

\begin{proof}
For ease of notation, we will let $k=2$. The general case is analogous. Let:
\[ \mathcal{P}_r^{(i)} = \langle a_1, \cdots, a_{n_i} \mid a_1^{p_i}, \cdots, a_{n_i}^{p_i}, [a_1^r,a_2], \{ [a_i,a_j] : i < j, (i,j) \ne (1,2) \} \rangle\]
for $r \in \Z$ with $(r,p_i)=1$. This is a presentation for $G_i$ and, since the homotopy type of $\mathcal{P}_r^{(i)}$ can be shown to depend only on $r \mod p_i$, we can take $r \in (\Z/p_i)^\times$.

It was shown by Browning \cite{Br79} (see also \cite[Proposition 9.2]{GL91}) that, if $\mathcal{P}$ is a presentation for $(\Z/p_i)^{n_i}$, then $X_{\mathcal{P}} \simeq X_{\mathcal{P}_r^{(i)}} \vee \ell S^2$ for some $r \in (\Z/p_i)^\times$, $\ell \ge 0$.
It suffices to show that $X_{\mathcal{P}_{r}^{(1)}} \vee X_{\mathcal{P}_{s}^{(2)}} \simeq X_{\mathcal{P}_{1}^{(1)}} \vee X_{\mathcal{P}_{1}^{(2)}}$ for all $r \in (\Z/p_1)^\times$, $s \in (\Z/p_2)^\times$.

As in \cref{prop:def(G*H)}, there exists integers $r_i$, $q_i$ such that $(q_i,q_j)=1$ for all $i \ne j$ and such that $r_i^{p_i}-1=p_iq_i$ and $r_i \equiv 1 \mod p_i$ for all $i$.
Let $r,s$ be integers such that $(r,p_1)=1$ and $(s,p_2)=1$. If $(rq_1,sq_2)=1$ then, by the same argument as given in \cref{prop:def(G*H)}, $G = G_1 \ast G_2$ has a presentation:
\begin{align*} \mathcal{P}_{r,s} = &\, \langle a_1, \dots, a_{n_1},  b_1, \cdots, b_{n_2} \mid \{a_i^{p_1}\}_{i=2}^{n_1}, \{b_i^{p_2}\}_{i=2}^{n_2}, a_1^{p_1} \cdot b_1^{p_2}, \\ 
& a_2(a_1^r)a_2^{-1}(a_1^r)^{-r_1},  b_2(b_1^s)b_2^{-1}(b_1^s)^{-r_2}, \{[a_i,a_j], [b_i,b_j] : i < j, (i,j) \ne (1,2) \} \rangle.
\end{align*}
This form is general for all $r, s$ since, by Dirichlet's theorem on arithmetic progressions, there exists $r',s'$ such that $r' \equiv r \mod p_1$, $s '\equiv s \mod p_2$ and $(r'q_1,s'q_2)=1$.

Let $(\mathcal{P}_{r,s})_+$ denote the presentation $\mathcal{P}_{r,s}$ with the additional relation $a_1^{p_1}$.
In $(\mathcal{P}_{r,s})_+$, we can replace $a_1^{p_1} \cdot b_1^{p_2} \leadsto b_1^{p_2}$ by left multiplying with $a_1^{-p_1}$, then replace $a_2(a_1^r)a_2^{-1}(a_1^r)^{-r_1} \leadsto [a_2, a_1^r]$ by right multiplying with $a_1^{r_1-1}$ (which works since $r_i \equiv 1 \mod p_i$), and similarly $b_2(b_1^s)b_2^{-1}(b_1^s)^{-r_2} \leadsto [b_2, b_1^s]$. This implies that $(\mathcal{P}_{r,s})_+$ and $\mathcal{P}_{r}^{(1)} \ast \mathcal{P}_{s}^{(2)}$ are $Q$-equivalent and so we have:
\[X_{\mathcal{P}_{r,s}} \vee S^2 \simeq X_{(\mathcal{P}_{r,s})_+} \simeq X_{\mathcal{P}_{r}^{(1)}} \vee X_{\mathcal{P}_{s}^{(2)}}.\]

Note that $\mathcal{P}_{1,s}$ differs from $\mathcal{P}_{r,s}$ by changing $a_2a_1a_2^{-1}a_1^{-r_1} \leadsto a_2(a_1^r)a_2^{-1}(a_1^r)^{-r_1}$. Since both relations hold in $G$, we can add $a_2(a_1^r)a_2^{-1}(a_1^r)^{-r_1}$ to $\mathcal{P}_{1,s}$ and add $a_2a_1a_2^{-1}a_1^{-r_1}$ to $\mathcal{P}_{r,s}$ to get that $X_{\mathcal{P}_{1,s}} \vee S^2 \simeq X_{\mathcal{P}_{r,s}} \vee S^2$. By symmetry, we also have that $X_{\mathcal{P}_{r,1}} \vee S^2 \simeq X_{\mathcal{P}_{r,s}} \vee S^2$ and so $X_{\mathcal{P}_{r}^{(1)}} \vee X_{\mathcal{P}_{s}^{(2)}} \simeq X_{\mathcal{P}_{1}^{(1)}} \vee X_{\mathcal{P}_{1}^{(2)}}$.
\end{proof}

The following can be found in \cite[Theorem 1.2 (3)(iv)]{Li93}. This can also be deduced by combining the earlier work \cite[Proposition 9]{SD79} with \cite[Theorem 1.7]{Br79}.

\begin{lemma} \label{lemma:abelian-non-cancellation}
Let $G = (\Z/p)^n$ for $p$ prime and $n \ge 1$. Let $\delta(G)$ denote the number of $\Aut(G)$-isomorphism classes of modules $\pi_2(X_{\mathcal{P}})$ for $\mathcal{P}$ a presentation with $\Def(\mathcal{P}) = \Def(G)$. If $p=2$, then $\delta(G) = 1$ and, if $p$ is odd, then:
\[ \delta(G) = \begin{cases} (\frac{p-1}{2},n-1), & \text{if $n$ is even} \\
 (\frac{p-1}{2}, \frac{n-1}{2}), & \text{if $n$ is odd}.
 \end{cases}
 \]
\end{lemma}

\begin{proof}[Proof of \cref{thm:non-uniqueness}]
Let $k \ge 2$ and, for $i = 1, \cdots, k$, let $p_i$ be distinct primes with $p_i \equiv 1 \mod 4$ and let $G_i = (\Z/p_i)^3$. By \cref{lemma:abelian-non-cancellation}, we have that $\delta(G_i) = 2$ and so there exists presentations $\mathcal{P}^{(i)}$, $\mathcal{Q}^{(i)}$ for $G_i$ such that $\Def(\mathcal{P}^{(i)}) = \Def(\mathcal{Q}^{(i)}) = \Def(G)$ and $\pi_2(X_{\mathcal{P}^{(i)}}) \not \cong \pi_2(X_{\mathcal{Q}^{(i)}})$ are not $\Aut(G_i)$-isomorphic.
	
Similarly to the proof of \cref{thm:non-existence}, $\pi_2(X_{\mathcal{P}^{(i)}}), \pi_2(X_{\mathcal{Q}^{(i)}}) \in \Omega_3^{G_i}(\Z)$ are minimal by \cite[Proposition 2.1]{Sw65}. This implies that $\Z G_i \nmid \pi_2(X_{\mathcal{P}^{(i)}}), \, \pi_2(X_{\mathcal{Q}^{(i)}})$ for all $i$.
By \cref{prop:abelian-collapse}, we have that 
\[ X_{\mathcal{P}^{(1)}} \vee \cdots \vee X_{\mathcal{P}^{(k)}} \simeq X_{\mathcal{Q}^{(1)}} \vee \cdots \vee X_{\mathcal{Q}^{(k)}}\] 
and so $\pi_2(X_{\mathcal{P}^{(1)}} \vee \cdots \vee X_{\mathcal{P}^{(k)}}) \cong \pi_2(X_{\mathcal{Q}^{(1)}} \vee \cdots \vee X_{\mathcal{Q}^{(k)}})$, as required.	
\end{proof}

\section{The unstable classification of smooth 4-manifolds}
\label{section:4-manifolds}

The aim of this section will be to discuss applications of stably free $\Z G$-modules to the unstable classification of smooth 4-manifolds. 
Whilst we will restrict our attention to 4-manifolds, it also possible to use the examples in \cref{thm:main} for $n \ge 3$ to study the unstable classification of $2n$-manifolds. For brevity, we will not discuss this here.

All manifolds will be assumed to be smooth and connected but not necessarily closed. We will let $\cong$ denote homeomorphism and let $\dcong$ denote diffeomorphism. 

\subsection{Boundary of thickenings construction} \label{ss:doubles}

Let $X$ be a finite $2$-complex. By \cite[Statement 7.2]{Co73}, $X$ is simple homotopy equivalent to a finite simplicial $2$-complex $X'$. By a general position argument, there exists an embedding $i : X' \hookrightarrow \R^5$ and a smooth regular neighbourhood $N(i)$ of this embedding which is unique up to diffeomorphism (see, for example, \cite[Section II]{KS84}).
We define $M(X) := \partial N(i)$, which is a closed smooth 4-manifold.
This depends a priori on the choice of $X'$ and embedding $i$, but we will omit these choices from the notation. For convenience, we will refer to $M(X)$ as a \textit{model} when we have fixed some choices of $X'$ and $i$ to obtain a well-defined manifold.

Recall that two closed smooth $4$-manifolds $M_1$, $M_2$ are \textit{smoothly $s$-cobordant}, written $\cong_{\sCob}$, if there exists a smooth $5$-manifold with boundary $W$ such that $\partial W \dcong M_1 \sqcup M_2$ and, for $i=1,2$, the induced inclusion maps 
\[ \iota_i : M_i \hookrightarrow \partial W \hookrightarrow W
\] 
are simple homotopy equivalences.
If two closed smooth 4-manifolds are smoothly $s$-cobordant, then they are simple homotopy equivalent and also stably diffeomorphic (see \cite[Theorem 3.4]{KPR22}).

The following is implicit in \cite{Wa66} (see also \cite[p15]{KS84} and \cite[Proposition 5 \& 6]{BC+21}). This gives a sense in which the construction $X \mapsto M(X)$ is well-defined.

\begin{prop} \label{prop:doubles-well-def}
Let $X$ and $Y$ be finite $2$-complexes such that $X \simeq_s Y$. Then, for any models $M(X)$ and $M(Y)$, we have $M(X) \cong_{\sCob} M(Y)$. In particular, they are stably diffeomorphic.
\end{prop}

The following special case will be useful later on.

\begin{lemma} \label{lemma:M-stabilised}
Let $X$ be a finite $2$-complex. Then there exists models $M(X)$ and $M(X \vee S^2)$ such that $M(X \vee S^2) \dcong M(X) \# (S^2 \times S^2)$.
\end{lemma}

\begin{proof}
Let $X'$ be a finite simplicial 2-complex such that $X \simeq_s X'$, let $i : X \hookrightarrow \R^5$, let $N(i)$ be a smooth regular neighbourhood of $i$ and let $M(X) = \partial N(i)$.

We have $X \vee S^2 \simeq_s X' \vee \Delta$ where $\Delta$ is a triangle and $X'$ and $\Delta$ are wedged at a $0$-simplex so that $X' \vee \Delta$ is a finite simplicial 2-complex. By embedding $\Delta$ in a sufficiently small neighbourhood of the wedge point, we can extend $i$ to an embedding $i_+ : X' \vee \Delta \hookrightarrow \R^5$ so that $N(i_+)=N(i) \natural (S^2 \times D^3)$ is a smooth regular neighbourhood of $i_+$ where $\natural$ denotes the boundary connected sum. We then take $M(X \vee S^2) = \partial N(i_+)$ and so $M(X \vee S^2) \dcong \partial N(i) \# \partial (S^2 \times D^3) \dcong M(X) \# (S^2 \times S^2)$.
\end{proof}

\subsection{Proof of \cref{thm:S+S*-intro}} \label{ss:proof-manifolds}

The aim of this section will be to prove the following theorem from the introduction, which we restate here for convenience.

\begingroup
\renewcommand\thethm{\ref{thm:S+S*-intro}}
\begin{thm}
Let $G$ be a finitely presented group such that $\gd(G)=2$ and suppose there exists a stably free $\Z G$-module $S$ of rank $k$ which is geometrically realisable and such that $S \oplus S^*$ is not a free $\Z G$-module.
Then both $\M^{\Diff}(G)$ and $\M(G)$ fail cancellation at level $k$. 
\end{thm}
\endgroup

We will begin by recalling the basic algebraic topology of $M(X)$ for $X$ a finite $2$-complex.

\begin{lemma} \label{lemma:pi_2-of-double}
Let $X$ be a finite $2$-complex with $\pi_1(X) \cong G$. Then $M(X)$ is a closed smooth $4$-manifold such that $\pi_1(M(X)) \cong G$ and there is an isomorphism of $\Z G$-modules:
\[ \pi_2(M(X)) \cong H^2(X;\Z G) \oplus H_2(X;\Z G). \]
\end{lemma}

\begin{proof}
This follows from the argument given in \cite[Section II]{KS84} in the case where $G$ is finite. The general case was proven in \cite[Theorem 4.2]{Ha09} (see also \cite[Lemma 5.7]{HH19}).
\end{proof}

\begin{lemma} \label{lemma:pi_2-cd=2}
Let $G$ be a finitely presented group such that $\cd(G)=2$ and let $X$ be a finite $2$-complex with $\pi_1(X) \cong G$. Then there are isomorphisms of $\Z G$-modules: 
\[ H_2(X ; \Z G) \cong \pi_2(X), \quad H^2(X;\Z G) \cong \pi_2(X)^* \oplus H^2(G;\Z G).\] 
Furthermore, $\pi_2(X)$ and $\pi_2(X)^*$ are stably free $\Z G$-modules.
\end{lemma}

\begin{proof}
The first part follows from the fact that $H_2(X;\Z G) \cong H_2(\wt X) \cong \pi_2(X)$ (see, for example, \cite[p81]{HMS93}), and holds for any $G$ finitely presented. Since $\cd(G)=2$, this is stably free by \cite[Lemma 5.4]{HH19}.
By the universal coefficient spectral sequence (see \cite[p11]{HH19}) applied to $X$, there is an exact sequence:
\[ 0 \to H^2(G;\Z G) \to H^2(X;\Z G) \to \pi_2(X)^* \to H^3(G;\Z G) \to 0. \]
Since $\cd(G)=2$, we have $H^3(G;\Z G)=0$. The dual of a stably free module is stably free (the proof coincides with that of \cref{prop:dual-of-proj} (i)) and so $\pi_2(X)^*$ is stably free. This implies it is projective and so the exact sequence splits and so we obtain the required isomorphism of $\Z G$-modules:
\[ H^2(X;\Z G) \cong \pi_2(X)^* \oplus H^2(G;\Z G). \qedhere \]
\end{proof}

\begin{proof}[Proof of \cref{thm:S+S*-intro}]
By assumption, there exists a finite 2-complex $X$ such that $\pi_1(X) \cong G$ and $\pi_2(X) \cong S$. Since $G$ is finitely presented with $\gd(G)=2$, there exists a finite aspherical 2-complex $Y_0$. Let $k$ be the rank of $S$ as a stably free $\Z G$-module and let $Y = Y_0 \vee k S^2$. 
By \cref{lemma:M-stabilised}, there exists models $M(Y)$ and $M(Y_0)$ such that $M(Y) \dcong M(Y_0) \# k(S^2 \times S^2)$. Fix a model $M(X)$. By \cite[Theorem 13]{Wh39}, there exists $r \ge 0$ such that $X \vee rS^2 \simeq_s Y \vee rS^2$ and so, by \cref{prop:doubles-well-def}, we have that $M(X) \# t(S^2 \times S^2) \dcong M(Y) \# t(S^2 \times S^2)$ for some $t \ge r$. 

We will now prove that $M(X) \not \simeq M(Y)$. First note that, by combining \cref{lemma:pi_2-of-double,lemma:pi_2-cd=2}, we obtain isomorphisms of $\Z G$-modules:
\[ \pi_2(M(X)) \cong S \oplus S^* \oplus H^2(G;\Z G), \quad \pi_2(M(Y)) \cong \Z G^{2k} \oplus H^2(G;\Z G). \]

Since $S$ is stably free and so projective, we have that $S^{**} \cong S$ by \cref{prop:dual-of-proj} (ii). Since $\cd(G)=2$, we have $H^2(G;\Z G)^* = 0$ \cite[Proposition 4.6]{HH19}. This gives:
\[ \pi_2(M(X))^* \cong S^* \oplus S^{**} \oplus H^2(G;\Z G)^* \cong S \oplus S^*, \quad \pi_2(M(Y)) \cong (\Z G^{2k})^* \oplus H^2(G;\Z G)^* \cong \Z G^{2k}. \]
If $M(X) \simeq M(Y)$ then, by the discussion at the start of \cref{section:module-invariants}, $\pi_2(M(X))$ and $\pi_2(M(Y))$ are $\Aut(G)$-isomorphic. By \cref{prop:dual-of-Aut(G)}, this implies that $\pi_2(M(X))^*$ and $\pi_2(M(Y))^*$ are $\Aut(G)$-isomorphic. 
In particular, for some $\theta \in \Aut(G)$, there are isomorphisms of $\Z G$-modules:
\[ S \oplus S^* \cong  \pi_2(M(X))^* \cong (\pi_2(M(Y))^*)_\theta \cong (\Z G^{2k})_\theta \cong \Z G^{2k}.\]
This contradicts the hypothesis that $S \oplus S^*$ is non-free. Hence $M(X) \not \simeq M(Y)$, as claimed.

Now let $\equiv$ denote either homeomorphism or diffeomorphism. Let $\ell \ge 0$ be minimal for which $M(X) \# \ell (S^2 \times S^2) \equiv M(Y') \# \ell (S^2 \times S^2)$. Since $M(X) \not \simeq M(Y')$ implies $M(X) \not \equiv M(Y)$, we have $\ell \ge 1$.
Let $M := M(X) \# (\ell-1)(S^2 \times S^2)$, $N := M(Y) \# (\ell-1)(S^2 \times S^2)$ and $N_0 := M(Y_0) \# (\ell-1)(S^2 \times S^2)$. By minimality of $\ell$, we have $M \not \equiv N$ and $M \# (S^2 \times S^2) \equiv N \# (S^2 \times S^2)$. Since $M(Y) \equiv M(Y_0) \# k(S^2 \times S^2)$, we have that $N \equiv N_0 \# k(S^2 \times S^2)$ by taking the connected sum with $(\ell-1)(S^2 \times S^2)$.
This shows that both $\M^{\Diff}(G)$ and $\M(G)$ fail cancellation at level $k$, as required.
\end{proof}

\section{Further directions}
\label{section:problems}

We will now collect together a list of open problems on unstable classification. The problems are arranged into lists (A) Finitely generated projective $\Z G$-modules, and (B) Finite 2-complexes.
The spirit of these problems is to search for examples which illustrate structural features of each respective classification, much like Wall's list of problems concerning finite 2-complexes \cite[List D]{Wa79}.

\renewcommand{\thesubsection}{\Alph{subsection}}
\renewcommand{\theplist}{\thesubsection\arabic{plist}}

\subsection*{Finitely generated projective $\Z G$-modules}
\setcounter{subsection}{1}

Recall that, for a projective $\Z G$-module $P$, the \textit{rank} is defined as 
\[\rank_{\Z G}(P) = \rank_{\Z}(\Z \otimes_{\Z G} P)\]
and the \textit{level} is $\ell(P) = \max\{m-n : P \oplus \Z G^n \cong Q \oplus \Z G^m, Q \in P(\Z G)\}$.

\begin{plist} \label{problem:proj-rank=0}
Do non-zero finitely generated projective $\Z G$-modules have non-zero rank?
\end{plist}

\begin{remark*}
This is true if $G$ is finite by Swan \cite{Sw60} and, more generally, provided $G$ satisfies Bass' strong conjecture on Hattori-Stallings rank. In particular, it holds if $(\Q,+) \not \le G$ \cite[Lemma 2.8]{Ev99b}.
This is true for stably free $\Z G$-modules over all groups $G$ since $\Z G$ is stably finite (see \cref{subsection:stably-finite}). In particular, if $G$ is torsion free and satisfies the Farrell-Jones conjecture, then $\wt K_0(\Z G)=0$ and so the statement holds (see \cref{problem:proj-t-free}).
It was pointed out by F. E. A. Johnson \cite{Jo23+} that, by the examples of Akasaki \cite{Ak82}, \cref{problem:proj-rank=0} has a negative answer in the case of infinitely generated projective modules. More specifically, for every non-solvable finite group $G$ there is an infinitely generated projective $\Z G$-module $P$ which is not free and for which $\Z \otimes_{\Z G} P = 0$.
\end{remark*}

\begin{plist} \label{problem:min-rank-in-c}
Does every stable class in $\wt K_0(\Z G)$ contain a projective $\Z G$-module of rank one? That is, do we have $\ell(P) = \rank_{\Z G}(P)-1$ for all finitely generated projective $\Z G$-modules $P$?
\end{plist}

\begin{remark*}
This is true for $G$ finite by Swan \cite{Sw60}. Since the zero class in $\wt K_0(\Z G)$ contains $\Z G$, this is also true for any group $G$ such that $\wt K_0(\Z G)=0$. For example, as above, this is true provided $G$ is torsion free and satisfies the Farrell-Jones conjecture.
\end{remark*}

\begin{plist} \label{problem:P(ZG)-multiple-stab}
For which $k \ge 2$ does there exist a group $G$ and finitely generated projective $\Z G$-modules $P$ and $Q$ such that $P \oplus \Z G^k \cong Q \oplus \Z G^k$ but $P \oplus \Z G^{k-1} \not \cong Q \oplus \Z G^{k-1}$? (see \cref{figure:further-branching}a)
\end{plist}

\begin{remark*}
This is open for all $k \ge 2$. Examples here would give further examples of the type considered in \cref{thm:simple-modules}. Presumably constructing examples which are stably free $\Z G$-modules would be most straightforward.
\end{remark*}

\begin{plist} \label{problem:P(ZG)-bound}
Does every finitely presented group $G$ have a cancellation bound for projective $\Z G$-modules? That is, does there exist a constant $d$ for which $P \oplus \Z G \cong Q \oplus \Z G$ implies $P \cong Q$ for finitely generated projective $\Z G$-modules $P$ and $Q$ of rank $\ge d$? (see \cref{figure:further-branching}b)
\end{plist}

\begin{remark*}
As explained in the introduction, examples do not exist when $\Z G$ is Noetherian (such as if $G$ is polycyclic-by-finite).
Examples were constructed in \cref{thm:main-SF-further} for $G=\ast_{i=1}^\infty T$, which is not finitely presented.
It would be interesting to know whether or not this example can be modified to give an example over a finitely presented group. For example, $G=\ast_{i=1}^\infty T \cong \ast_{i\in \Z} T$ is a subgroup of the finitely presented group $(\ast_{i \in \Z} T) \rtimes \Z$, where $\Z$ freely permutes the copies of $T$, which is isomorphic to $T \times C_\infty$. However, the non-free stably free modules of \cref{thm:main-SF-further} all become free upon passage to $\Z[T\times C_\infty]$ via extension of scalars. We are indebted to Sam Hughes for discussions on this point.
\end{remark*}

\begin{figure}[h] \vspace{-4mm} 
\begin{center}
\begin{tabular}{ccccc}
\begin{tabular}{l}
\begin{tikzpicture}
\draw[fill=black] (1,0) circle (2pt);
\draw[fill=black] (1.5,1) circle (2pt);
\draw[fill=black] (3,0) circle (2pt);
\draw[fill=black] (2.5,1) circle (2pt);
\draw[fill=black] (2,2) circle (2pt);
\draw[fill=black] (2,3) circle (2pt);

\node at (2,3.6) {$\vdots$};
\draw[black] (2.75,1) node[right]{
$\left.
    \begin{array}{ll}
           \\ \\ \\
    \end{array}
\right \} k$};
\node at (1,3) {(a)};

\draw[thick] (1,0) -- (1.5,1) -- (2,2) -- (2,3)  (3,0) -- (2.5,1) -- (2,2);
\end{tikzpicture} 
\end{tabular}
&&&&
\begin{tabular}{l}
\begin{tikzpicture}
\draw[white] (1,0) node[right]{$\ell = 0$};

\draw[fill=black] (-0.25,0) circle (2pt);
\draw[fill=black] (0.5,0) circle (2pt);
\draw[fill=black] (2,0) circle (2pt);
\draw[fill=black] (2,1) circle (2pt);
\draw[fill=black] (2,2) circle (2pt);
\draw[fill=black] (2,3) circle (2pt);
\draw[fill=black] (1.25,0) circle (2pt);
\draw[fill=black] (1.25,1) circle (2pt);
\draw[fill=black] (0.5,1) circle (2pt);

\node at (1.2,3) {(b)};
\node at (2,3.8) {$\vdots$};
\draw[black] (2,1) node[right]{
$\left.
    \begin{array}{ll}
           \\ \\ \\
    \end{array}
\right \} d$};
\draw[thick] (-0.25,0) -- (0.5,1)--(2,2) (0.5,0) -- (0.5,1) (2,0) -- (2,1) 
(2,1) -- (2,2) -- (2,3) (1.25,0)--(2,1) (1.25,1)--(2,2);
\end{tikzpicture}
\end{tabular}
\end{tabular}
\end{center}
\caption{Further branching phenomena}
\label{figure:further-branching}
\vspace{-2mm}
\end{figure}

\begin{plist} \label{problem:S+S*}
Does there exists a group $G$ and finitely generated projective $\Z G$-modules $P$ and $Q$ such that $P \oplus \Z G \cong Q \oplus \Z G$ but $P \oplus P^* \not \cong Q \oplus Q^*$?
\end{plist}

\begin{remark*}
By Swan \cite{Sw60}, there are no examples when $G$ is finite. This is motivated by \cref{thm:S+S*-intro}. Given these applications, the main case of interest is therefore the case where $G$ is finitely presented and of type $\FL$, and where $P$ is stably free.
\end{remark*}

\begin{plist} \label{problem:proj-t-free}
Does there exist a torsion free group $G$ and a finitely generated projective $\Z G$-module which is not stably free?
\end{plist}

\begin{remark*}
This is a well known problem and appeared, for example, in Wall's problem list \cite[Problem A1]{Wa79}.
Note that the Farrell-Jones conjecture is a broad generalisation of this question. Many torsion free groups $G$ are known to satisfy the Farrell-Jones conjecture (see \cite{LR05}) and so the question above has a negative answer in these cases.
Whilst the examples obtained in this paper are all stably free, they do at least demonstrate that more elaborate projective modules exist in the case of torsion free groups.	
\end{remark*}

\subsection*{Finite 2-complexes}
\setcounter{plist}{0}
\setcounter{subsection}{\value{subsection}+1}

Recall that a CW-complex $X$ is irreducible if $X \simeq Y \vee Z$ for CW-complexes $Y$, $Z$ implies that $Y$ or $Z$ is contractible.

\begin{plist} \label{problem:complexes-irred}
Let $X_i$, $Y_i$ be irreducible non-simply connected finite $2$-complexes. When does 
\[X_1 \vee \cdots \vee X_k \simeq Y_1 \vee \cdots \vee Y_k\] 
imply that $X_i \simeq Y_{\sigma(i)}$ for some $\sigma \in S_k$?
\end{plist}

\begin{remark*}
This is motivated by the results in \cref{section:induced}. Here irreducibility is necessary since it rules out the following two situations:
\begin{clist}{(a)}
\item \textit{Exchange of subfactors}: If $X \not \simeq Z$, $Y \not \simeq \ast$, then $(X \vee Y) \vee Z \simeq X \vee (Y \vee Z)$.
\item \textit{Non-cancellation}: If $X \vee S^2 \simeq Y \vee S^2$, $X \not \simeq Y$, then $X \vee (Z \vee S^2) \simeq Y \vee (Z \vee S^2)$. 
\end{clist}

The finite $2$-complexes given in the proof of \cref{thm:non-uniqueness} are irreducible and so show that some further conditions must be imposed.
This was shown to be true by Jajodia \cite[Corollary 4]{Ja79} in the case where the $X_i$, $Y_i$ have a single $2$-cell.
\end{remark*}

\begin{plist} \label{problem:complexes-mult}
For which $k \ge 1$ do there exist finite $2$-complexes $X_1$, $X_2$ with $\pi_1(X_1) \cong \pi_1(X_2)$ such that $X_1 \vee kS^2 \simeq X_2 \vee kS^2$ and $X_1 \vee (k-1)S^2 \not \simeq X_2 \vee (k-1)S^2$? (see \cref{figure:further-branching}a)
\end{plist}

\begin{remark*}
This is open for all $k \ge 1$. The question was asked in the case $k=2$ can be found in \cite[Problem C]{Dy79a} and later appeared in \cite[p124]{HMS93}. Following the same method of proof of \cref{thm:main}, one imagines that examples in \cref{problem:P(ZG)-multiple-stab} which are stably free could lead to examples here.
\end{remark*}

\begin{plist} \label{problem:complexes-bound}
Does there exist a finitely presented group $G$ such that, for infinitely many $k \ge 0$, there are homotopically distinct finite $2$-complexes at level $k$? (see \cref{figure:further-branching}b)
\end{plist}

\begin{remark*}
This is the analogue of \cref{problem:P(ZG)-bound} for finite 2-complexes. As with \cref{problem:complexes-mult}, examples there which are stably free could lead to examples here. Note that this is equivalent to asking that, for infinitely many $k \ge 0$, there are homotopically distinct finite $2$-complexes $X_1$, $X_2$ with $\pi_1(X_i) \cong G$ and $\chi(X_i) = k + \chi_{\min}(G)$.
\end{remark*}

\section*{Conflict of interest}

 On behalf of all authors, the corresponding author states that there is no conflict of interest.

\bibliography{stablyfree.bib}
\bibliographystyle{amsplain}

\end{document}